\documentclass[11pt]{article}
\usepackage{tikz}
  \usepackage{latexsym}
  \usepackage{comment}
  \usepackage[font={small,it}]{caption}
 
  \usepackage{amsmath,amsthm,amsfonts,amssymb,graphicx,epsfig,latexsym,color,amssymb,xcolor}
  \usepackage{MnSymbol,wasysym}
  
  \usepackage{epsf,enumerate}
  \usepackage{ulem}
   \usepackage[utf8]{inputenc}
    \usepackage{BOONDOX-cal}
        \usepackage{float}

  \def\rank{\operatorname{rank}}
  \def\im{\operatorname{Im}}

  \newtheorem{thm}{Theorem}[section]
    \newtheorem{theorem}{Theorem}[section]
  \newtheorem{lemma}[thm]{Lemma}
  \newtheorem{claim}[thm]{Claim}
  \newtheorem{cor}[thm]{Corollary}
  \newtheorem{prop}[thm]{Proposition}
    \newtheorem*{obs}{Observation}
  {}
 
 \usepackage{amsthm}

\makeatletter
\newtheorem*{rep@theorem}{\rep@title}
\newcommand{\newreptheorem}[2]{
\newenvironment{rep#1}[1]{
 \def\rep@title{#2 \ref{##1}}
 \begin{rep@theorem}}
 {\end{rep@theorem}}}
\makeatother

\newreptheorem{theorem}{Theorem}
\newreptheorem{lemma}{Lemma}
 
  \theoremstyle{remark}
  \newtheorem{example}[thm]{Example}
  
  \newtheorem{definition}[thm]{Definition}
  \newtheorem{remark}[thm]{Remark}
  
  \newtheorem*{definition*}{Definition}
  \newtheorem*{remark*}{Remark}

    \newcommand{\msr}[1]{}
      \newcommand{\msb}[1]{}

      \newcommand{\mb}[1]{}

  \def\R{{\mathbb R}}
  \def\Rbb{{\mathbb R}}

  \def\Z{{\mathbb Z}}

  \def\L{{\mathcal L}}
   \def\cH{{\mathcal H}}
     \def\cU{{\mathcal U}}
  
  \def\P{{\mathcal P}}
  
  \def\I{{\mathcal I}}

  \def\hull{\operatorname{Hull}}

  \def\C{{\mathcal C}}
    \def\h{{\mathfrak h}}
   \def\hh{\hat{\mathfrak h}}
    \def\o{\overline}

\begin{document}
  
  \normalem
  \title{Topological median algebra structures on ER homology manifolds I: local cubulation}
  \author{Mladen Bestvina, Kenneth Bromberg and Michah Sageev}
 \maketitle
 \begin{abstract}
 We study topological median algebra structures on Euclidean spaces
 and, more generally, ER homology manifolds. We show that all such
 median structures have a local CAT(0) cubulation structure. We also
 show that topological median algebra structures are completely metrizable as
 median metric spaces if and only if intervals are compact. We give
 examples of both metrizable and non-metrizable such structures, as
 well as provide a construction for producing many non-locally
 cubulated topological median algebra structures on the unit ball in
 Euclidean space.
 \end{abstract}

Median algebras have played an important role in recent years in generalizing the notion of a CAT(0) cube complex. Generally speaking, the study of median metric spaces (which are median algebras) and coarse median spaces (coarse versions of median metric spaces) have played analogous roles to what $\R$-trees and quasi-trees played to simplicial trees. 
Topological median algebras are more general than median metric spaces in that they do not require a metric, just a topology on the median algebra. The overarching question we address in this paper is this: given a topological space $X$, can one describe all the topological median algebra structures on $X$?  In this paper, we give an answer to this question when $X=\R^n$ and more generally, when $X$ is an ER homology manifold.  

One natural class of examples for topological median algebra structures on $\R^n$ is given by a CAT(0) cubulation of it. That is, a CAT(0) cube complex whose underlying space is $\R^n$. We will see that while this is not generally true, it is true locally.  While the goal theorem was for $\R^n$ it turns out that the inductive part of the proof brings us into the more general setting of ER-homology manifolds, so that the theorem is actually proved for them. 

\begin{reptheorem}{LocalCubingHigherDimension} Every topological median algebra structure on an ER homology manifold  is locally isomorphic to a finite CAT(0) cube complex as a topological median algebra. 
\end{reptheorem}

The above theorem is generalization of a theorem of Bowditch \cite{Bowditch2018}, who proved a local cubulation theorem for complete median metric structures on $\Rbb^n$. Thus, a natural question following the above theorem, is  whether or not the topological median algebra structures in our setting are completely metrizable. That is, does there exists a complete median metric structure which realizes the given topological median algebra structure. To this end, we prove

\begin{reptheorem}{CompleteMetrization} A topological median algebra structure on an ER homology manifold is completely metrizable as median metric space if and only if all median intervals are compact. 
\end{reptheorem}

We will also see that there are many examples of topological median
algebra structures on $\R^n$ which do not come from cubulations of
$\R^n$, even in dimension 2. Moreover, we construct examples of
natural topological median algebra structures on $\R^n$ which are not
completely metrizable as median metric spaces. 

Finally, one can ask whether all topological median algebra structures on any ``reasonable" topological space are locally cubulated as in Theorem \ref{LocalCubingHigherDimension}. We show, using a median variant of the Floyd construction, that in dimensions $n>1$, there exist infinitely many complete median metric structures on the closed $n$-ball which are not locally cubulated.

The structure of the paper is as follows. In Sections \ref{MedPrelim}, we collect and develop the necessary background tools we need about median algebras. In Section \ref{Fancy}, we discuss the results about $Z$-sets and homology manifolds that we will need, including Mitchell's theorem that the boundary of a homology manifold is a homology manifold and Bredon's local constancy of the orientation sheaf. In Section  \ref{TopMedAlg}, we discuss some results we will use about topological median algebras in our setting.  
In Section \ref{DimensionTwo}, we give a proof of the local cubulation theorem in dimension two. While not logically necessary, the 2-dimensional argument introduces all the relevant ideas and is self contained, in that it does not rely on ER homology manifold theory. In Section \ref{HigherDimensions} we give the general proof for ER homology manifolds. In Section \ref{Metrization}, we discuss the proof of Theorem \ref{CompleteMetrization} and along the way give the description of the topological median algebra structure as an exhaustion by compact CAT(0) cube complexes with subdivision. 

A word about terminology in this paper. For brevity,  we will use the term \emph{median structure} for a topological median algebra structure on a given space. 

In a sequel to this paper, we will show a converse to the local cubulation theorem. Namely, we show that a local cubulation together with a natural nesting condition on the leaves of the cubulation, gives rise to a median structure. 

{\bf Acknowledgements.} We thank Elia Fioravanti for his interest and
for a careful reading of the first draft. The first author was
partially supported by NSF grant DMS-2304774 and by a grant from the
Simons Foundation (SFI-MPS-SFM-00011601). The second author was partially
supported by NSF grants DMS-1906095 and DMS-2405104 and a grant from the Simons foundation (SFI-MPS-SFM-00006400). The last author was supported by ISF grant 660/20. 
 
  \section{Median algebra preliminaries}
  \label{MedPrelim}

  \subsection{Basic definitions}
  We recall here some median algebra basics. We refer the reader to Bowditch's recent extensive reference work on median algebras \cite{Bowditch24} for more details.

  A median algebra is a set $M$ together with a ternary operation $(a,b,c)\mapsto abc$ such that 
  
  \begin{enumerate}
  \item (Majority rules) $aab=a$ for $a,b\in M$
  \item (Symmetry) $abc=bac=acb$ for all $a,b,c\in M$
  \item (Distributivity) $ab(cde)=(abc)(abd)e$ for all $a,b,c,d,e\in M$
  \end{enumerate}
  
Another axiom which is sometimes used instead of the distributivity axiom is the associativity 
axiom:

$$(axb)xc=ax(bxc) \text{ for all } a,b,c,x\in M$$
The two axioms are equivalent, although the proof is non-trivial (see
\cite{Bowditch24}). A consequence of the axioms is the {\it long}
distributive law:
$$ab(xyz) = (abx)(aby)(abz)$$
which is sometimes useful.

Given $a,b\in M$, the  \emph{interval} between $a$ and $b$ is defined as $[a,b]=\{x\vert abx=x\}$. Note that from the  distributive law we have

\begin{equation}\label{retraction}
    ab(abx)=(aba)(abb)x=abx \tag{$\star$}
  \end{equation}

This computation tells one has a map $\rho=\rho_{ab}:M\to [a,b]$, defined by $\rho(x)=abx$, and that this map is a retraction. This map is called the \emph{gate map} to the interval. The long distributive law implies that the gate map is a median homomorphism.

We say that a subset $C\subset M$ is convex if for any $a,b\in C$, we have $[a,b]\subset C$. The most obvious convex sets are intervals themselves.

\begin{lemma}
If $c,d\in [a,b]$, then $[c,d]\subset[a,b]$
\end{lemma}

\begin{proof}
Given $x\in[c,d]$, we have: $x\in [c,d]\implies cdx=x \implies abx=ab(cdx)=(abc)(abd)x=cdx=x \implies x\in [a,b]$
\end{proof}

\begin{lemma}[Preimage of convex is convex]
If $C\subset [a,b]$ is a convex subset of an interval and $\rho:M\to [a,b]$ is the gate map, then $\rho^{-1}(C)$ is convex. 
\label{ConvexPreimages}
\end{lemma}

\begin{proof}
Suppose that $x,y\in\rho^{-1}(C)$. Let $x'=\rho(x)$ and $y'=\rho(y)$. Then for any $z\in [x,y]$, we have 

$$\rho(z)=abz=ab(xyz)=(abx)(aby)z=x'y'z \subset [x',y']\subset C$$
\end{proof}

\subsection{Sholander's Theorem}

Consider three points $a,b,c\in M$, then there are three intervals defined by $a,b$ and $c$: $[a,b]$, $[b,c]$, and $[a,c]$. From  (\ref{retraction}), we have that the median $abc$ is contained in all three intervals. In fact, we have something stronger, namely that $m$ is the only point common to all three intervals.

\begin{lemma}
If $a,b,c\in M$, then $[a,b]\cap[b,c]\cap[a,c]=\{abc\}$. 
\end{lemma}

\begin{proof}
Let $m=abc$. Suppose $x\in[a,b]\cap [b,c] \cap [a,c]$. Then we have 

$$x=abx=ab(bcx)=(abc)bx=mbx=mb(acx)=(mba)(mbc)x=mmx=m$$
\end{proof}

The content of Sholander's Theorem is that this property of triple intersection -- along with two other innoucous properties which are satisfied for median algebras --  is enough to determine a median algebra structure. 

\begin{thm} (Sholander \cite{Sholander54})
Suppose that $\P(M)=2^M$ and that we have a set $M$ and a map $I:M\times M \to \P(M)$. Suppose further that this map satisfies the following properties

\begin{enumerate}
\item $I(a,a)=\{a\}$
\item if $c,d\in I(a,b)$, then $I(d,c)\subset I(a,b)$
\item if $a,b,c\in M$, then $\vert I(a,b)\cap I(b,c)\cap I(a,c) \vert =1$
\end{enumerate}

Then the map $M^3\to M$ defined by $(a,b,c) \mapsto I(a,b)\cap I(b,c)\cap I(a,c)$ defines a median algebra structure on $M$. 
\end{thm}

Properties of intervals can be useful for proving things for median algebras as well. For example, we have the following

\begin{obs}
 If $M$ is a median algebra and $c\in [a,b]$, then  $[a,c]\cap [c,b]= \{c\}$
\end{obs}

\begin{proof}
If $c\in [a,b]$, then any point in the double intersection $[a,c]\cap [b,c]$ is in the triple intersection $[a,b]\cap [a,c]\cap [b,c]$, which is a singleton. 
\end{proof}

Another elementary property of intervals is that the intersection of two intervals is an interval.

\begin{lemma}
The intersection of two intervals in a median algebra is either empty or an interval.
\label{IntervalIntersection}
\end{lemma}

\begin{proof}

Let $[a,b]$ and $[c,d]$ be two intervals.
Let $\hat{c}=abc$ and $\hat{d}=abd$, the  projections of $c$ and $d$ onto $[a,b]$.
We wish to show that $[\hat c, \hat d]=[a,b]\cap [c,d]$.

The containment $[a,b]\cap [c,d]\subset [\hat c,\hat d]$ is straightforward:

$$x\in [a,b]\cap [c,d]\implies abx=x, cdx=x$$
$$\implies \hat c \hat dx=(abc)(abd)x=ab(cdx)=x\implies x\in [\hat c, \hat d]$$

For the other containment: it is clear from the definition  that $[\hat c,\hat d] \subset [a,b]$. What is not immediately clear is that $[\hat c,\hat d] \subset [c,d]$. In fact, this containment does not hold when $[a,b]\cap [c,d]=\emptyset$. 

So suppose there exists $x \in [a,b]\cap [c,d]$. Since $\hat c$ is the gate of $c$ in $[a,b]$, we have that $\hat c$ is contained in the interval between $c$ and any element  of  $[a,b]$. In particular, $\hat c\in [c,x]$. Since $x \in [c,d]$, we thus have that $\hat c \in [c,d]$. Similarly $\hat d\in [c,d]$. So we have that from convexity that $[\hat c,\hat d] \subset [c,d]$, as required.
\end{proof}

 \subsection{Cubes and dimension }
  
The 2 point set $\{0,1\}$ has a canonical median algebra structure dictated by the majority rules axiom. The product of any two median algebras can be endowed naturally with a median algebra structure with coordinate-wise operations. The product median structure on  $\{0,1\}^n$ is called a \emph{median $n$-cube}. The \emph{rank} of a median algebra $M$, denoted $rk(M)$, is the largest $n$ for which there exists  a median $n$-cube embedded in $M$. 

Rank is additive for products. That is, for $M$ and $N$ finite rank
median algebras, $rk(M\times N)=rk(M)+rk(N)$ (see
\cite{Bowditch24}, Lemma 8.2.2).

  For a median square $\{0,1\}^2$, we use the notation $\square(a_1,a_2,a_3,a_4)$ for the points of the median square arranged cyclically: $a_1=(0,0), a_2=(1,0), a_3=(1,1)$ and $a_4=(0,1)$. We then obtain
 the median structure $a_{i-1}a_ia_{i+1}=a_i$  ($i$ taken mod 4). The term \emph{median square} in $M$ will also be taken to mean that the $a_i$'s are distinct. Sometimes we will need to address \emph{degenerate median squares} which means that  $a_{i-1}a_ia_{i+1}=a_i$, but that some of the $a_i$'s are equal. We have the following description of how a median square can degenerate. 
  
  \begin{obs}
  If $\{a_1, a_2,a_3,a_4\}$ is a degenerate median square then, up to cyclic permutation of the $a_i$'s, we have the following possibilities:
  \begin{enumerate}
  \item $a_1=a_2=a_3=a_4$, that is the square is a single point
  \item $a_1=a_2$ and $a_3=a_4$, but no other equalities. That is, the square is the interval $[a_1,a_2]$
  \end{enumerate}
   In particular, if three of the points of a median square are distinct, then the median square is non-degenerate. 
  \label{DegenerateSquare}
  \end{obs}

  \begin{proof}
  Suppose we are not in case 1 and suppose that $a_1=a_2$. Then $a_3=a_2a_3a_4=a_1a_3a_3=a_4$. We also need to rule out diagonal points being equal. So suppose we have $a_1=a_3$. then we would have $a_2=a_1a_2a_3=a_1a_2a_1=a_1$ and similarly $a_2=a_3$, and from this we would have that all four points are equal, and this would put us in case 1. 
  \end{proof}

For a median square $\square(a,b,c,d)$, we see that $a,b,c,d\in [a,c]$. Since $[a,c]$ is convex and contained in any convex set containing $a$ and $c$, we see that $[a,c]=\hull\{a,b,c,d\}$. A basic fact is that a median square has a natural product structure.

\begin{lemma}
For a median $\sigma=\square(a,b,c,d)$, there is a natural isomorphism $\hull(\sigma)=[a,c]\cong [a,b]\times [a,d]$. 
\label{Bromberg}
\end{lemma}

\begin{proof}

We consider the maps:

\begin{align*}
    \phi: [a,b]\times [a,d] &\to [a,c] \\        
        (y,z)\ \quad &\mapsto yzc
\end{align*}
\begin{align*}
\psi:[a,c] &\to [a,b]\times [a,d]  \\
x &\mapsto (abx, abx).
\end{align*}
We need to check that these maps are inverses of one another. 

First, observe $\rho_{ab}([a,d])=a$: if $z\in[a,d]$ then $abz=ab(adz)=(aba)(abd)z=aaz=a$. Similarly, $\rho_{ad}([a,b])=a$. Now we compute:
$$\phi\circ\psi(x)= (abx)(adx)c=ax(bdc)=axc=x$$
$$\psi\circ\phi(y,z)=(ab(yzc), ad(yzc))=((aby)(abz)c, (ady)(adz)c) $$
$$= ((aby)ac,a(adz)c=(aby, adz)=(y,z)$$

\end{proof}

Given a subset $S\subset M$, we define the convex hull of $S$, $\hull(S)$, to be the intersection of all convex sets containing $S$. Note that intervals are convex, and that the convex hull of two points $a,b$ is $[a,b]$.

\begin{obs}
Let $C=\{0,1\}^n$ be a standard median cube. Let $\overline 0 =
(0,\ldots,0)$ and $\overline 1= (1,\ldots,1)$. Then $C=[\overline
  0,\overline 1]$. 
\end{obs}

\begin{proof}
This is a straightfoward calculation:  For  $M=\{0,1\}$, we have $0x1=x$ for all $x\in\{0,1\}$, and since the median in a product is computed componentwise, we have that  $\overline{0}x \overline{1}=x$ for all $x\in\{0,1\}^n$.
\end{proof}

From this we obtain the following. 

\begin{cor}
If $M$ is any median algebra and $C$ is a median cube in $M$, then $\hull(C)=[a,b]$, where $a$ and $b$ are diagonally opposite vertices of $C$.
\label{CubeIsAnInterval}
\end{cor}

Lemma \ref{Bromberg} can be generalized to median cubes.

\begin{cor}
Let $f:\{0,1\}^n\hookrightarrow M$ be a median cube $C$. Let $a=f(0,\ldots,0)$, $b=f(1,\ldots,1)$ and for $i=1,\ldots,n$, let $a_i=f(0,\ldots,0,\stackrel{i}{1},0,\ldots,0)$. Then $[a,b]=\hull(C)\cong \prod_i [a,a_i]$. 
\label{BrombergNdim}
\end{cor}

\begin{proof}
The proof is by induction on $n$. We know it is true for $n=2$. For larger $n$, we consider the median square $\square(a,b,c,d)$ in $M$ spanned by the following 4 points: 

$$a, b=f(1,\ldots,1,0), c=f(1,\ldots,1), a_n= f(0,\ldots,0,1)$$
The first two are the diagonal points of a median cube of one lower dimension. So by Lemma \ref{Bromberg} we have that 
$C=[a,b]\times[a,a_n]$ and by induction $[a,b]=\prod_{i=1}^{n-1} [a,a_i]$. This gives the claim 
\end{proof}

In the situation of Corollary \ref{BrombergNdim}, the pairs $\{a,a_i\}$'s are called  the \emph{edges} of the median cube $C$ and we say that these edges \emph{span} $C$. 

The following basic lemma lies at the heart of the connection between median structures and the flag condition in CAT(0) condition for cube complexes: namely that a collection of edges at a point that pairwise span squares actually spans a cube. 

\begin{lemma}
Let $a, a_1,\ldots,a_n\in M$ be distinct points in $M$ such that each pair $\{a_i,a_j\}$ span a median square with $a$. That is, there exists $a_{ij}\in M$ such that $\square(a,a_i,a_{ij},a_j)$ is a median square. Then there exists a median cube spanned by the $\{a,a_i\}$'s. 
\label{Flag}
\end{lemma}

  \begin{lemma}\label{rectangles are squares}
    If $\square(a,b,c,d)$ and $\square(x,y,c,d)$ are median squares then so is
    $\square(a,b,y,x)$. 
  \end{lemma}
  
  \begin{proof}
  By symmetry, it suffices to check $aby=b$:
  
  $$aby=(dab)(cab)y=ab(cdy)=abc=b.$$
  \end{proof}

  \begin{proof}[Proof of Lemma \ref{Flag}]
    We argue by induction on $n$, with the case $n=2$ being
    vacuous. Suppose $n>2$. We refer to Figure \ref{Figure:Flag} for this discussion. 
    
    \begin{figure}[h]
\begin{center}
\includegraphics[scale=.7]{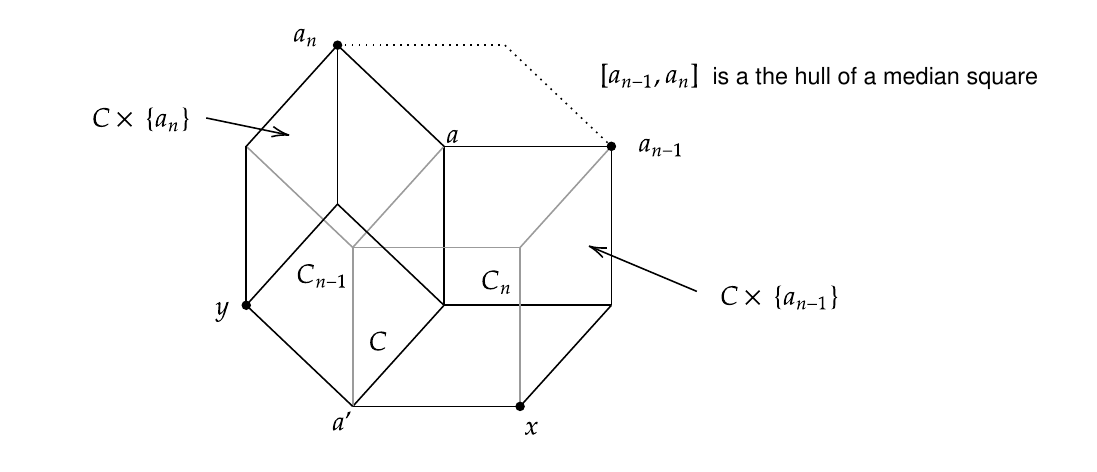}
\end{center}
\caption{Two median $n$-cubes $C_{n-1}$ and $C_n$ meet along an $(n-1)$-cube $C$ to give a bigger cube. The diagonals provide a median square $\square(a_{n-1},a_n,y,x)$, which has a structure of a product of an $(n-2)$-cube and a square, which gives an $n$-cube.}
\label{Figure:Flag}
\end{figure}

    In the median cube $C$ spanned by the edges  $\{a,a_1\},\cdots,\{a,a_{n-2}\}$, let $a'$ be the vertex diagonally opposite to
    $a$.  By Lemma \ref{Bromberg}, the  hull of the median cube $C_n$ spanned by $\{a,a_1\},\cdots,\{a,a_{n-1}\}$ is the product
    $\hull(C_n)=\hull(C)\times [a,a_{n-1}]$. In these product coordinates, we have an $(n-1)$-cube $C\times \{a_{n-1}\}$, which is the $(n-1)$-face of $C_n$ opposite $C$. Let $x$ denote the diagonally opposite point to $a_n$ in $C\times\{a_{n-1}\}$.
    Similarly, the hull of the median cube $C_{n-1}$  spanned by
  $\{a,a_1\}, \cdots, \{a,a_{n-2}\},\{a,a_n\}$ is the product $\hull(C_{n-1})=\hull(C)\times [a,a_n]$,
  we have an $(n-1)$-face $C\times \{a_n\}$ opposite to $C$ in $C_{n-1}$, and we let $y$ denote the diagonally opposite point to $a_{n-1}$ in $C\times\{a_{n-1}\}$. 
But the four vertices  $\{a_{n-1},a_n,y,x\}$ form a median square by Lemma \ref{rectangles
      are squares}, since the median squares $\square(a_{n-1},x,a',a)$ and $\square(a_n,y,a',a)$ share a median edge $\{a',a\}$.
    Now by Lemma \ref{Bromberg} we have that the hull of this median square
    is the product $\hull(C)\times [a_{n-1},a_n]$, which is hull of the desired
    $n$-cube, since $\hull(C)$ is the hull of a median $(n-2)$-cube and $[a_{n-1},a_n]$ is the hull of a median square.
  \end{proof}

We will need the following lemma which produces median squares.

\begin{lemma}[Double projection]
Suppose that $a,b,c,d$ are distinct elements of $M$, such that $abc=b$
and $bcd=c$. Let $\hat d = \rho_{cd}(a)$ and $\hat a=\rho_{ab}(d)$. Then $\hat a,b,c, \hat d$ form a (possibly degenerate) median square.
\label{DoubleProjection}
\end{lemma}

\begin{figure}[h]
\begin{center}
\includegraphics{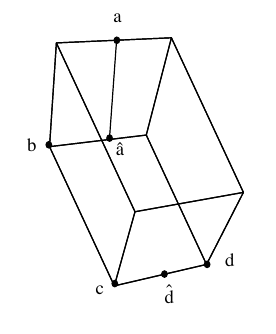}
\end{center}
\caption{Projecting $a$ onto $[c,d]$ and $d$ onto $[a.b]$ to get a median square.}
\label{Figure:DoubleProjection}
\end{figure}

\begin{proof}
We refer to Figure \ref{Figure:DoubleProjection} for this discussion. Since $\rho_{bc}([a,b])=b$, and $\hat a \in [a,b]$, we have that $\hat abc=b$. Similarly,
$bc\hat d=c$. We are left to show $b\hat a\hat d=\hat a$ and $c\hat d \hat a = \hat d$. We show the first and the second is the same. It is another application of the distributive law:

$$\hat a \hat d b= (adb)(adc)b=ad(bbc)=adb=\hat a.$$

\end{proof}

Finally, we state a technical, but useful lemma characterizing the situation in which the half open interval is not convex.

\begin{lemma}
   Let $a,b\in M$. Then $(a,b]$ is not convex if and only if there exists a median square in $[a,b]$ with $a$ as a vertex. 
   \label{SquareInInterval}
   \end{lemma}   
   
   \begin{proof}

   \msb{changed order of points to around square to match next paragraph}
   First, note that if $[a,b]$ contains a median square $\sigma= \square (a,c,e,d)$. Then $\sigma=[c,d]$ contains $a$ and thus $\sigma\not\subset (a,b]$. Since $c,d\in (a,b]$ it follows that $(a,b]$ is not convex. 
      
   For the other direction, assume that $(a,b]$ is not convex. Then there exist points $c,d\in (a,b]$ such that $[c,d]\not\in(a,b]$. Since $[c,d]\subset[a,b]$, it follows that $a\in [c,d]$. Thus we have $cda=a$. We define $e=bcd$. We claim 
   that $\{a,c,e,d\}$ form a median square. We already have $cda=a$, so we need to show 
   \begin{itemize}
   \item \underline{$cde=e$:} $cde=cd(bcd)=bcd=e$
   \item \underline{$eca=c$:} $eca=(bcd)ca=bc(dca)=bca=c$
   \item \underline{$eda=d$:} same as $eca=c$, by symmetry
   \end{itemize}
   
   Now we need to show that $a,b,c,d$ are distinct. By our choice of $c$ and $d$ we immediately have that $c\not=a, d\not=a$ and $c\not =d$. Thus, by Observation \ref{DegenerateSquare} we have that $\square (a,c,d,e)$ is the required median square.
   \end{proof}

    \section{ER's, Z-sets, homology manifolds}
    \label{Fancy}

In this section we will address some topological notions which we will employ in this paper. 
We let $X$ denote a separable, metrizable, locally compact topological space. 
In addition, we will assume that the
covering dimension $\dim X<\infty$. It is a standard fact that such
$X$ embeds as a closed subset of some $\R^N$ and it has an exhaustion
by compact subsets (see for example \cite{Munkres75}, Theorem 50.5 applied to the 1-point compactification of X.)

\subsection{ER's and ENR's}
We recall the following definitions.

\begin{definition}
  $X$ is an ENR (Euclidean Neighborhood Retract) if for some (every)
  embedding $X\subset\R^N$ as a closed set there is a neighborhood of
  $X$ that retracts to $X$. In addition, $X$ is an ER (Euclidean
  Retract) if all of $\R^N$ retracts to $X$.
\end{definition}

The following facts are standard:
\begin{itemize}
  \item $X$ is an ENR if and only if it is locally contractible \cite {Liao49}
    \item Every ENR is homotopy equivalent to a CW complex and in
      particular Whitehead's theorem holds for ENRs (see e.g. \cite{milnor}).
\item
  $X$ is an ER if and only if it is a contractible ENR.
  \end{itemize}

\subsection{Z-sets}
\begin{definition}
  Let $X$ be an ENR. A closed subset $Z\subset X$ is a {\it Z-set} in
  $X$ if for every open set $U\subset X$ inclusion $U\smallsetminus
  Z\hookrightarrow U$ is a homotopy equivalence.
\end{definition}

Equivalently, there is a homotopy $f_t:X\to X$ for $t\in [0,1]$ such
that $f_0=id$ and $f_t(X)\cap Z=\emptyset$ for $t>0$. We will use the
following form of this statement.

\begin{lemma} [{\cite[Theorem 2.3]{torunczyk}}]\label{perturb}
  Let $X$ be an ENR and $Z\subset X$ a Z-set. Let $K$ be a finite simplicial
  complex and $L\subset K$ a subcomplex. Fix a metric on $X$ and let
  $f:(K,L)\to (X,X\smallsetminus Z)$ be a map. Then for every
  $\epsilon>0$ there is a homotopy $f_t:K\to X$ such that $f_0=f$, $f_t=f$ on
  $L$, $f_1(K)\subset X\smallsetminus Z$ and $d(f(x),f_t(x))<\epsilon$
  for every $x\in X,t\in [0,1]$. 
\end{lemma}

The following proposition is well-known. When $X$ is compact it is
proved for example in \cite[Theorem B.1]{BestvinaHorbez19}. The
noncompact case is a straightforward generalization.

\begin{prop}
  Let $X$ be an ENR  and $Z\subset X$ a closed
  subset. Suppose that for every $z\in Z$ and every open set $V\subset
  X$ with $z\in V$ there exists an open set $W\subset V$ with $z\in W$
  such that inclusion $W\smallsetminus Z\hookrightarrow
  V\smallsetminus Z$ is nullhomotopic. Then
  $Z$ is a Z-set in $X$.
  \label{LocalConditionZset}
\end{prop}

\subsection{Homology manifolds}
We use singular homology with coefficients in $\Z$.

\begin{definition}
  Let $X$ be an ENR. We say that $X$ is a {\it homology $n$-manifold}
  if for every $x\in X$ we have $H_i(X,X\smallsetminus\{x\})=0$ for
  $i\neq n$ and $H_n(X,X\smallsetminus\{x\})\cong\Z$. We say that $X$
  is a {\it homology $n$-manifold with boundary} if for every $x\in X$
  we have $H_i(X,X\smallsetminus\{x\})=0$ for
  $i\neq n$ and $H_n(X,X\smallsetminus\{x\})=0$ or $\cong\Z$. The set
  $\partial X$
  of points $x\in X$ where $H_n(X,X\smallsetminus\{x\})=0$ is the {\it
    boundary} of $X$.
\end{definition}

We will use the following theorem of Mitchell \cite{Mitchell90}:

\begin{thm}\label{mitchell}
  If $X$ is a homology $n$-manifold with boundary then $\partial X$
  is a homology $(n-1)$-manifold (with empty boundary).
\end{thm}

Homology manifolds in dimension $\leq 2$ are manifolds (Wilder
\cite{wilder}), but in higher dimensions this is not true. Typical
examples are the one-point compactification of the Whitehead manifold
in dimension 3, or the suspension of a homology sphere in dimensions
$>3$. We will also need the following theorems of Bredon.

\begin{thm} [Invariance of Domain \cite{BredonSheaf}]
  Let $X$ be a homology $n$-manifold and $A\subset X$ a subset which
  is also a homology $n$-manifold. Then $A$ is an open subset of $X$.
  \label{InvarianceOfDomain}
\end{thm}

\begin{thm} [Local Orientability \cite{Bredon69}]
\label{bredon}
  Let $X$ be a homology $n$-manifold. Then the local homology sheaf
  $\{U\mapsto H_n(X,X\setminus U)\}$ is locally constant. More explicitly,
  if $c$ is a relative $n$-cycle representing $k\in\Z$ times a generator of
  $H_n(X,X\smallsetminus x)$ then there is a neighborhood $U$ of $x$
  such that for all $y\in U$, $c$ represents $k$ times a generator of
  $H_n(X,X\smallsetminus y)$.

\end{thm}

A homology $n$-manifold $X$ is {\it orientable} if there is a choice
of a generator $u_x\in H_n(X,X\smallsetminus x)$ for all $x\in X$
which is locally constant: every $x$ admits a relative cycle
representing $u_x$ so that the same cycle represents $u_y$ for $y$ in
a neighborhood of $x$.

\begin{cor}
  Every homology $n$-manifold is locally orientable, i.e. every point
  has an orientable neighborhood.
\end{cor}

We will need:

\begin{prop}
  Let $X$ be an ENR and $Z\subset X$ a Z-set. Assume that
  $X\smallsetminus Z$ is a homology $n$-manifold. Then $X$ is a
  homology $n$-manifold with boundary and $\partial X=Z$ is a homology
  $(n-1)$-manifold.
  \label{Zset}
\end{prop}

\begin{proof}
  We need to check that $H_i(X,X\smallsetminus z)=0$ for all $z\in Z$
  and all $i$ and then the statement will follow from Theorem
  \ref{mitchell}. Let $c$ be a relative cycle representing a class in
  $H_i(X,X\smallsetminus z)$. Thus $c$ is a map $c:(K,K_0)\to
  (X,X\smallsetminus z)$ where $K$ is a finite complex and $K_0$ a
  subcomplex with $c|K_0$ representing $\partial c$. Use Lemma
  \ref{perturb} (with $L=\emptyset$) to homotope $c:K\to X$ to
  $c':K\to X$ by a homotopy so small that it keeps the image of $K_0$
  disjoint from $\{z\}$, and so that the image misses $Z$. Thus
  $[c]=[c']=0$.
\end{proof}

Another proposition we need is

\begin{prop}\label{LocallyEverything}
Let $X$ be an ENR homology $n$-manifold and $a\in X$. Suppose $P\subset X$
is a finite simplicial complex and assume:
\begin{itemize}
\item $a\in P$,
  \item every simplex in $P$ that contains $a$ is contained in an
    $n$-simplex, and
    \item every $(n-1)$-simplex that contains $a$ is contained in
      exactly two $n$-simplices.
\end{itemize}
Then $P$ contains a neighborhood of $a$.
\end{prop}

\begin{proof}
  We may assume, by ordering the vertices of $P$, 
   that $P$ is also a
  $\Delta$-complex, i.e. the set of vertices of every simplex is
  compatibly ordered.
  First note that every open $n$-simplex $\sigma$ is an open set in $X$ by
  Invariance of Domain. In particular, $\sigma$ represents a generator
  of $H_n(X,X\smallsetminus x)$ for every $x$ in the interior of
  $\sigma$. Now suppose that we have two $n$-simplices
  $\sigma,\tau$ that share an $(n-1)$-face $\eta$. Subtracting all
  faces except for $\eta$ from $\sigma\cup\tau$ produces an
  $n$-manifold, which must also be open in $X$ by Invariance of
  Domain. 

  By subdividing $P$ and passing to a subcomplex we may assume that
  $P$ is contained in an orientable neighborhood $U$ of $a$, and
  choose an orientation of $U$. Consider the chain
  $c=\sum\epsilon_\sigma \sigma$, where the sum is over all
  $n$-simplices $\sigma$ and $\epsilon_\sigma=\pm 1$ depending on whether the
  orientations of $\sigma$ coming from the ordering and the
  orientation of $U$ agree or disagree. It now follows that in the
  chain $\partial c$ all $(n-1)$-simplices that contain $a$ cancel, in
  other words $c$ represents a local homology class in
  $H_n(X,X\smallsetminus x)$ for $x$ in a neighborhood of $a$. By the
  local constancy of the sheaf this class is $k$ times the preferred
  generator coming from the orientation. But for points near $a$
  inside an open $n$-simplex we see that $k=1$. In particular, $P$
  must contain a neighborhood of $a$.
  Indeed, if $x\in U\setminus P$, then since $P$ is closed (because it
  is compact) there exists an open $V\subset U$, a neighborhood of $x$
  such that $P$ does not meet $V$. Therefore the restriction (by
  excision) of $c$ to to $H_n(X,X\smallsetminus\{x\})$ is trivial and
  so by Theorem \ref{bredon} $x$ does not belong to a fixed neighborhood of $a$.
\end{proof}

\section{Topological median algebras}
\label{TopMedAlg}

A \emph{topological median algebra} is a median algebra $M$ with a topology so that the median map $M^3\to M$ is continuous. As noted above, we will simply use the term \emph{median structure} for a topological median algebra structure on a given space. 

In this paper, we will address median structures
on nice spaces like $\R^n$ and more generally, ER homology manifolds. 

To begin with, we will assume that $M$ is Hausdorff.
Note that the retraction to an interval $[a,b]$ defined above ($x\to
abx$) is continuous, so that intervals in topological median algebras
are topological retracts of the ambient space. Thus, intervals are
always closed, and if $M$ is path-connected so are all intervals.

Intervals may not be compact (see Example \ref{NonCompactExample}) . However, in locally compact spaces, sufficiently small
intervals are compact. In fact, we have the following local convexity
result, which we shall employ. We say $M$ is \emph{locally convex} if
every point has an arbitrarily small convex neighborhood (not
necessarily open).

\begin{thm}(\cite[12.2.4 and 12.2.5]{Bowditch24})
If $M$ is connected, locally compact and finite rank, then $M$ is
locally convex. In fact, every point has arbitrarily small compact
convex neighborhoods. 
\label{LocallyConvex}
\end{thm}

\subsection{Standing assumptions on $M$}

We restrict our considerations to median structures on $M$ that are homeomorphic to
a closed subset of an ER homology manifold. This is equivalent to
assuming that $M$ is locally compact, metrizable, separable and finite
dimensional. Thus we do not consider spaces like the long line or the
Hilbert cube. In addition, we will assume that $M$ is path connected
to avoid discrete median structures like the vertex set of a $CAT(0)$ cube
complex. By Proposition \ref{EmbeddedCube} below, $M$ will have
finite rank.

\begin{prop}\label{ER}
  Under these assumptions, if $M$ supports a median structure then $M$
  is a Euclidean retract (ER).
\end{prop}

In the proof we will need the following fact, due to Elia Fioravanti (see
\cite{Bowditch24}).

\begin{prop}[Fioravanti]\label{elia}
  If a path connected space $M$  admits a median structure, we have $\pi_n(M,*)=0$ for
  $n>0$ and any basepoint $*$.
  \label{Fioravanti}
\end{prop}

\begin{proof}
  Let
  $f:(S^n,*)\to (M,*)$ be any map. We have a factorization of $f$ as
  $q\Delta$ where $\Delta:S^n\to S^n\times S^n$ is the diagonal map
  and $q(a,b)=f(a)f(b)*$. Now the claim follows since $\Delta$ is
  homotopic into the $n$-skeleton $S^n\times *\cup *\times S^n$ on
  which $q$ is constant.
\end{proof}

\begin{proof}[Proof of Proposition \ref{ER}]
  If $\dim M=n$, then $M$ is an ENR iff it is locally $n$-connected ($LC^n$),
  i.e. for every $x\in M$ and every neighborhood $U$ of $x$ there is a
  neighborhood $V\subset U$ such that inclusion $(V,x)\hookrightarrow (U,x)$
  is trivial in $\pi_i$ for $i=0,1,\cdots,n$ (see e.g.
  \cite[Theorem 7.1]{hu}). Given $x\in U$ we apply Theorem
  \ref{LocallyConvex} to find a compact neighborhood $N\subset U$ of
  $x$. Letting $V=\mathring N$ be the interior of $N$ and using the fact that $\pi_i(N)=0$ for all
  $i$ (by Proposition \ref{Fioravanti}), we see that $M$ is $LC^n$ and hence an ENR. Now every ENR is
  homotopy equivalent to a CW complex so
  Whitehead's theorem combined with $\pi_i(M)=0$ for all $i$ implies
  that $M$ is contractible and hence an ER.
\end{proof}

Since convex subsets inherit the ambient median structure,  we
immediately obtain

\begin{cor}\label{LocallyClosed}
  Every locally compact convex subset of $M$ is also an ER.
\end{cor}

\subsection{Median structures on $\R$}
We start by establishing that $\R$ has a unique median structure.

\begin{prop}
  There is only one median structure on $\R$: if $a\leq x\leq b$ then
  $axb=x$.
\end{prop}

\begin{proof}
  The median interval $[a,b]$ is a retract of $\R$ so it is
  connected. Therefore if $a\leq x\leq b$, then $x\in [a,b]$ and so $axb=x$.
\end{proof}

The same argument works to show that any real tree has a unique
topological median structure (by a \emph{real tree} we mean a
Hausdorff, locally connected topological space, such that any two points are joined by a
unique embedded arc). In this case, the median $abc$ is the center of
the tripod spanned by $a,b,c$. (Note that the local connectivity assumption above is necessary. For example, the Warsaw circle is uniquely arc-connected, yet it is not difficult to see that it admits no median structure. )

 \subsection{Rank 1 intervals}
 
 We say that an interval $[a,b]$ is \emph{rank 1} if it is rank  at most 1 as a median algebra (this allows for the case that the interval is a single point). This case is described for topological median algebras in \cite{Bowditch24} in the case that intervals are compact. We give a self-contained account here in our setting. 

We say that an interval $[a,b]$ is \emph{additive} if for every $x\in [a,b]$, we have that $[a,x]\cup[x,b]=[a,b]$. Recall that we always have $[a,x]\cup[x,b]\subset [a,b]$
 \begin{lemma}
 The following are equivalent. 
 
\begin{enumerate}
\item $[a,b]$ is rank 1.
\item $[a,b]$ is additive.
\item $[a,b]$ is an arc with endpoints $a$ and $b$. 
\end{enumerate}
\label{Additive}
 \end{lemma}
 
 \begin{proof}

$1 \implies 2$: Suppose that $[a,b]$ is not additive. Thus there
   exists $x\in [a,b]$ such that $[a,b]\not= [a,x]\cup[x,b]$. Choose
   $y\in [a,b]\setminus([a,x]\cup[x,b])$. Now we project $y$ onto each
   of the intervals: $u=ayx$ and $v=ybx$. We will
   show that $\{x,u,y,v\}$ is a median square, contradicting that
   $[a,b]$ is of rank 1.

     We first show that $\{x,u,y,v\}$ satisfy the conditions of a
     (possibly degenerate) median square.  

     \noindent\underline{$uxv=x$}: $uxv=(axy)x(bxy)=xy(abx)=xyx=x$
     
     \noindent\underline{$yux=u$}: $yux=xy(axy)=axy=u$

    \noindent\underline{$yvx=v$}: $yvx=xy(bxy)=bxy=v$
     
     \noindent\underline{$uyv=y$}: $uyv=(axy)y(bxy)=xy(aby)=xyy=y$
     
     Now we show the points $\{x,u,y,v\}$ are distinct. By Observation \ref{DegenerateSquare}, the possibilities are that all the vertices are equal, or that it degenerates so that $x=u$ and $y=v$, or $x=v$ and $u=y$. Thus it suffices to show that $x\not=u$ and $x\not=v$. 
     
     \noindent\underline{$x\not=u$}: If $x=u$, then $y=v$,  so that $y\in [x,b]$, contradicting our assumption.
    
     \noindent\underline{$x\not=v$}: If $x=v$, then $y=u$ so that $y\in [a,x]$, contradicting our assumption
  
$2\implies 3$: Since $[a,b]$ is Hausdorff and path-connected, there exists an arc (embedded path) $\alpha$ joining $a$ and $b$. We will argue that $\alpha=[a,b]$. 

For $x\in [a,b]$, from additivity, we have that $[a,b]=[a,x]\cup[x,b]$. Let $[a,x)=[a,x]-\{x\}$ and $(x,b]=[x,b]\setminus \{x\}$.  We claim that $[a,x)$ and $(x,b]$ are the path components of $[a,b]-\{x\}$. First note for any $y\in [a,x]$, with $y\not=x$, we have $[a,y]\subset [a,x)$. This is because we already have $[a,y]\subset [a,x]$, so in order for $[a,y]\not\subset [a,x)$, we need to have $x\in [a,y]$. But then by the definition of intervals, $x=axy=y$. 

Secondly, note that for any $y\in [a,x)$, the interval $[a,y]$ is path-connected (as already noted, it is a retract of $M$). Thus, $[a,x)$ is path-connected and similarly $(x,b]$ is path-connected. 
Note that since $[x,b]$ is closed in $[a,b]$, $[a,x)$ is open in $[a,b]$ and hence in $[a,b]\setminus \{x\}$

Similarly $(x,b]$ is open in $[a,b]\setminus
\{x\}$. We thus have that $[a,b]\setminus \{x\}$ is a union of two
path-connected disjoint open sets $[a,x)$ and $(x,b]$, which are
consequently the two path components of $ [a,b]\setminus
\{x\}$.

Now suppose that $x\in [a,b]\setminus \alpha$. The path $\alpha$ is then a path in $[a,b]\setminus \{x\}$ joining $a$ and $b$, a contradiction to the above. 

$3\implies 1$: This is trivial in light of Proposition \ref{EmbeddedCube}.
 \end{proof}

 \subsection{Embedding cubes}

By a \emph{median embedding} of a median algebra $A$ into $M$, we mean
an injective map preserving medians. We will see in this section that
under mild assumptions on the ambient topological median algebra,
median cubes give rise to topological embeddings of real cubes
(products of standard closed intervals in $\R$). A version of the following proposition can be found in \cite{Bowditch24} under slightly different assumptions. If $f:\{0,1\}^n\to M$ is an embedding of a median cube, then we say that $f$ extends to a \emph{face-respecting map}\msr{maybe it should be ``face-saving"} $\hat f:[0,1]^n\to M$ if there exists such a map and for every face $\sigma$ of $[0,1]^n$, $\hat f(\sigma)\subset\hull(f(\sigma^0))$ (where $\sigma^0$ is the 0-skeleton of $\sigma$).

\begin{prop}[Embedded cube]
Suppose that $f:\{0,1\}^n\to M$ is an embedding of a median cube in $M$.  Let $C=\im(f)$. 

\begin{enumerate}
\item The map $f$ extends to a continuous, face respecting  embedding $\hat f:[0,1]^n\to \hull(C)$.
\item Suppose further that the corresponding edges of the cube in $M$ are additive: for each $i=1,\ldots k$, the interval $$[f(0,\ldots,0), f(0,\ldots,0,\stackrel{i}{1},0,\ldots0)]$$ is additive.  
 Then the face-respecting embedding of $f:[0,1]^n\to M$
 extending $f$  is a 
 median embedding. This embedding is unique up to pre-composition with a median homeomorphism of $[0,1]^n$ preserving $\{0,1\}^n$ pointwise. 
\end{enumerate}
\label{EmbeddedCube}
\end{prop}

\begin{proof} For the first item, recall that a  path-connected, Hausdorff space is arc-connected. 
For $n=1$, the claim reduces to saying that each interval is arc-connected.  Given any two points $a,b\in M$,  choose any path between $a$ and $b$ in $M$, and apply the gate retraction $\rho_{ab}$ to this path. We can then replace the path by an arc joining $a$ and $b$, since $M$ and hence $[a,b]$ is Hausdorff, and get the required embedding $\hat f:[0,1]\to [a,b]$.

Now note that $C=f(C)$ is a median cube, so that by Corollary \ref{BrombergNdim}, we have that a natural isomorphism $h:\prod_i[a,e_i]\to \hull(C)$, where $e_i=f(0,\ldots,0,\stackrel{i}{1},0,\ldots,0)$. For each $i$, we use the above to obtain an embedding  $g_i:[0,1]\to [a,e_i]$ with $g_i(0)=a$ and $g_i(1)=e_i$.  We now define $g: [0,1]^n\to \prod_i[a,e_i]$ component-wise: $g(t_1,\ldots,t_n)=(g_1(t_1),\ldots,g_n(t_n))$, which is an embedding since each $g_i$ is an embedding. We then define $\hat f=h\circ g$. Since $g$ is defined component-wise, $\hat f$ is face respecting.

For the second item,  note that under the assumpion that that $[a,e_i]$ is an arc,   the  embedding  $g_i:[0,1]\to [a,e_i]$ with $g_i(0)=a$ and $g_i(1)=e_i$ and is clearly median preserving. Note that it is unique up to pre-composition by an order preserving homeomorphism of $[0,1]$ preserving $\{0,1\}$. The embedding of the cube $\hat f=h\circ g$ is then a median-preserving embedding, which is unique up pre-composition by  a median preserving homeomorphism of $[0,1]^n$
preserving $\{0,1\}^n$, as required.  
\end{proof}

\begin{remark} While we are assuming all the standard assumptions for $M$ listed at the beginning of our section, the only assumption we are really using in the proof is that $M$ is arc-connected. 
\end{remark}

 When there is a median embedding $f:[0,1]^k\to M$ where $$[f(0,\ldots,0), f(0,\ldots,0,\stackrel{i}{1},0,\ldots0)]$$ are all additive (as in Item 2 of Lemma \ref{EmbeddedCube}), we call $Im(f)$ a \emph{cube} in $M$. 
 
 We will need that two such intersect nicely.
 
 \begin{lemma}
 The intersection of two cubes of $M$ is a cube.
 \label{CubeIntersection}
 \end{lemma}
\begin{proof} Note that if  $x,y\in[0,1]^n$, with $x=(x_1,\ldots,x_n)$ and $y=(y_1,\ldots,y_n)$, the interval $[x,y]$ is itself a subcube, namely the convex hull of the median cube $\prod_i {x_i,y_i}$. 

Now if $\sigma$ and $\tau$ are two cubes in $M$, then by Lemma \ref{CubeIsAnInterval} there are $a,b,c,d\in M$ such that $\sigma=[a,b]$ and $\tau=[c,d]$. Then by Lemma \ref{IntervalIntersection}, their intersection is an interval $\sigma\cap\tau=[x,y]$. Note that since we have a median isomorphism of $\sigma$ with $[0,1]^k$, so we have $x,y\in[0,1]^k$, with $x=(x_1,\ldots,x_k)$ and $y=(y_1,\ldots,y_k)$, the interval $[x,y]$ is itself a subcube, namely the convex hull of the median cube $\prod_i {x_i,y_i}$. 
 \end{proof}

We then obtain the following, which we shall make use of later on. 

\begin{cor} For a median structure on any retract of $\R^n$ we have

\begin{itemize}
\item The rank is at most $n$. 
\item Every point has a base of compact, convex neighborhoods.
\end{itemize}
\label{RankAndConvexityEuclidean}
\end{cor}
    
 \begin{proof}
 The first item follows from Proposition \ref{EmbeddedCube} and the second item follows from finite rank, Proposition \ref{LocallyConvex} and local compactness. 
 \end{proof}

\subsection{Median structures on CAT(0) cube complexes}

One way to get median structures on $ER$ homology manifolds is via CAT(0) cubulations. The traditional connection between CAT(0) cube complexes and median algebras is that the vertex set of a CAT(0) cube complex is naturally a discrete median algebra (see e.g. \cite{Roller98} \cite{Nica04} \cite{ChatterjiNiblo04} \cite{Sageev14}). 
We recall here this connection and then show how the median structure on the vertices actually extends to a median structure on all of the CAT(0) cube complex. 

Let $X$ be a finite dimensional CAT(0) cube complex.
 Let  $\hat\cH$ denote the collection of hyperplanes of $X$, each of which separates into two halfspaces. The collection of halfspaces, we denote as $\cH$. Recall that the 0-skeleton $X^0$ of $X$ is in one-to-one correspondence with the collection of  $\cal U_\cH$ of DCC (descending chain condition) ultrafilters on $X$
$$ X^0 \to \cU_\cH$$
$$v\mapsto \alpha_v$$

Here the ultrafilter $\alpha_v$ is simply the collection of halfspaces which contain $v$. We have a median structure on $\cU_\cH$ defined in the usual way:

$$\alpha_u\alpha_v\alpha_w=(\alpha_u\cap\alpha_v)\cup(\alpha_v\cap\alpha_w)\cup(\alpha_w\cap\alpha_u)$$
 
 And this in turn defines a median structure on $X^0$.

 We now want to extend this median structure to all of $X$. We can do this in the following way. 
 
Leaves are generalizations of hyperplanes and can be either regular or singular.  We define a \emph{regular leaf} in $X$ as a subset of $X$ which is parallel to copy of a hyperplane, running through the interior of the same cubes. More precisely, if $\sigma$ is an $n$-cube in $X$, a \emph{midleaf} of $\sigma$ is an $n-1$ cube parallel to one of the faces of $\sigma$, intersecting the interior of $\sigma$, but not necessarily going through the barycenter of $\sigma$.

As with hyperplanes, a regular leaf is a maximal extension of a midleaf.  Note that since regular leaves are parallel to hyperplanes, they satisfy the same properties. Namely, a regular leaf is embedded, and separates $X$ into two components. The closure of one of these components is called a \emph{halfspace}. A \emph{singular} leaf is a limit of regular leaves. Thus, it may contain a vertex or indeed a cell of $\sigma$ (see Figure \ref{Leaves}). 
 
\begin{figure}[h]
\begin{center}
\includegraphics[scale=.7]{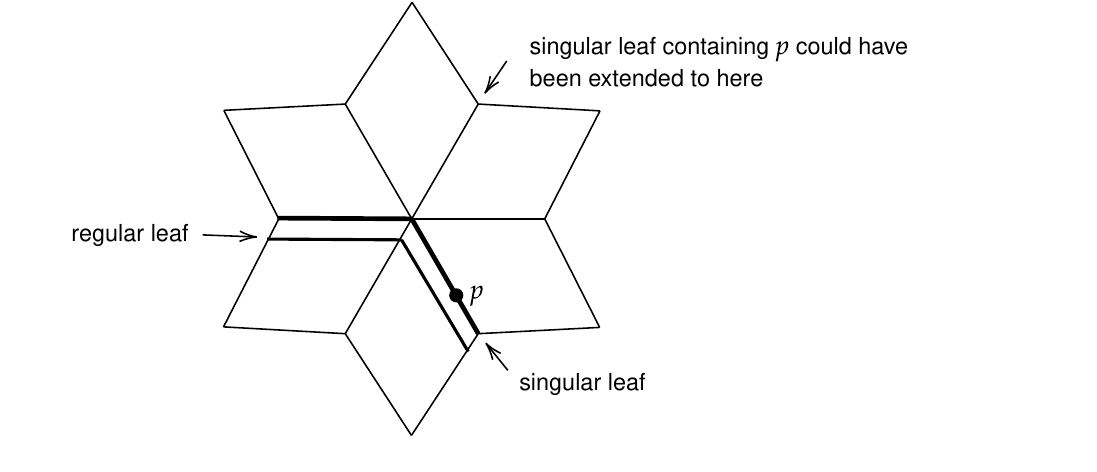}
\end{center}
\caption{Regular leaves and singular leaves.}
\label{Leaves}
\end{figure}

 We let $\hat\L$ denote the collection of all leaves, and $\L$ the collection of all halfspaces. Now we note that given 
 any point in $X$, we get an \emph{almost ultrafilter} on $\L$. Namely, if $x\in X$, we can define
 
 $$\alpha_x =\{\h\in\L\vert x\in\h\}$$
 
 Note that this consistently choses a side of every leaf, except for the finitely many leaves that contain $x$, for which both sides are chosen. Also, note that it satisfies an analogue of the DCC condition. Namely, any descending sequence of halfspace $\h_1\supset \h_2,...$ in $\alpha_x$ satisfies $\bigcap h_i\not=\emptyset$. 
 
 Conversely, if $\alpha$ is an almost ultrafilter satisfying the DCC condition, then there exists $x$ such that $\alpha=\alpha_x$. 
 
To see this, note that  $\alpha$ is a poset of subsets of $X$, partially ordered by inclusion. Since $\alpha$ satisfies DCC, every descending chain has a non-empty intersection. Also, note that when $\h_1\supset\h_2\supset...$ is a such an infinite chain, the intersection of all of them is a half space, so that every descending chain has a lower bound. Thus, by Zorn's lemma, there exists a minimal element $\h_0$ in $\alpha$, bounded by some leaf $\ell_0$. For all leaves disjoint from $\ell_0$, we then have that the co-orientation chooses the halfspace of $\ell$ containing $\h_0$. Now we apply induction. Since $\ell_0$ is a CAT(0) cube complex of one lower dimension, and its leaves are the leaves of $X$  intersecting $\ell_0$. Thus, viewed as an ultrafilter on the halfspaces of $\ell_o$, $\alpha=\alpha_x$ for some $x\in\ell_0$. We thus get that $\alpha=\alpha_x$ also as an ultrafilter on the halfspaces of $X$.

\section{Topological median algebra structures on
 ER homology manifolds}
    \label{MedER}
    Here we see how the above is applied to median structures on $\Rbb^n$ and its homological siblings. 
    We start with one of the technical statements that shows why we need to delve into ER homology manifolds in the first place. For this section, we let $M$ be an ER homology $n$-manifold with a median structure. 
 
   \begin{prop}
 Let $[a,b]$ be an interval and let $\rho: M\to [a,b]$ be the gate projection map. Suppose that $(a,b]$ is convex. Let $$F=\overline{\rho^{-1}((a,b])}\setminus \rho^{-1}((a,b])$$ Then $F$ is a closed convex ER homology  $(n-1)$-manifold. 
   \label{Fset}
   \end{prop}
   
   \begin{proof}
   Consider the following sets:
   
   $$U=\rho^{-1}((a,b]), \quad \overline U=\text{the closure of } U,\quad F=\o{U}-U$$
   
   We claim that all three of these sets are convex. The set $U$ is convex since $(a,b]$ is convex and the preimage under $\rho$ of a convex set is convex. It follows that $\o U$ is convex as it is the closure of a convex set. Finally, observe that $\{a\}$ is convex, so that $\rho^{-1}(a)$ is convex. Thus $F=\rho^{-1}(a)\cap \o U$ is the intersection of two convex sets, and so is also convex.

   How we argue that $F$ is a Z-set of $\o U$. By Proposition
   \ref{LocalConditionZset}, it suffices to show that for every open
   neighborhood $V$ of any $x\in F$, there exists a smaller
   neighborhood $W\subset V$ of $x$ such that $W\smallsetminus
   F\hookrightarrow V\smallsetminus F$ is nullhomotopic.
   
   Given $V$ as above, by Corollary \ref{RankAndConvexityEuclidean} we can find a
   convex neighborhood $N$ of $x$, such that $N\subset V$. Setting
   $W=\mathring N$ to be the interior of $N$, we have 
$$W\smallsetminus F\hookrightarrow N\smallsetminus F\hookrightarrow
   V\smallsetminus F.$$
   Since $N\smallsetminus F=N\cap U$ is convex and locally closed, it
   follows from Corollary \ref{LocallyClosed} that it is contractible,
   implying that $W\smallsetminus F\hookrightarrow V\smallsetminus F$
   is nullhomotopic as required.

   Now we apply Proposition \ref{Zset} to the sets $F\subset \o U$. Since
   $\o U$ has a basis of convex sets, it follows that $\o U$ is locally
   contractible and hence an ENR. Since it is also convex it is an ER by Corollary \ref{LocallyClosed}.

   Observe that $U=\o U\setminus F$ is an open set in $M$ and so by Lemma \ref{InvarianceOfDomain} is an  ER homology $n$-manifold. It then follows by Proposition \ref{Zset} that $\o U$ is a homology $n$-manifold with boundary and $\partial \o U=F$ is a homology $n-1$-manifold. 
 \end{proof}

  When $(a,b]$ is convex as in the proposition, we call the set $F$
    defined above \emph{the orthogonal manifold to $[a,b]$ at $a$}.

The following lemma justifies the term ``orthogonal".

\begin{lemma}
Let $(a,b]$ be a convex half-open interval in $M$. Let $F$ be the orthogonal manifold to $[a,b]$ at $a$. Let $c\in F\setminus \{a\}$ and let $V$ be a neighborhood of $c$. Then there exist $b'\in (a,b]$ and $c'\in (a,c]\cap V$ such that $a,b',c'$ are three vertices of a median square.  
\label{NearbySquare}
\end{lemma}
     
   \begin{proof}
   We refer to diagram \ref{NearbySquarePic} for the following. 
   \begin{figure}[hbt]
\begin{center}
\includegraphics{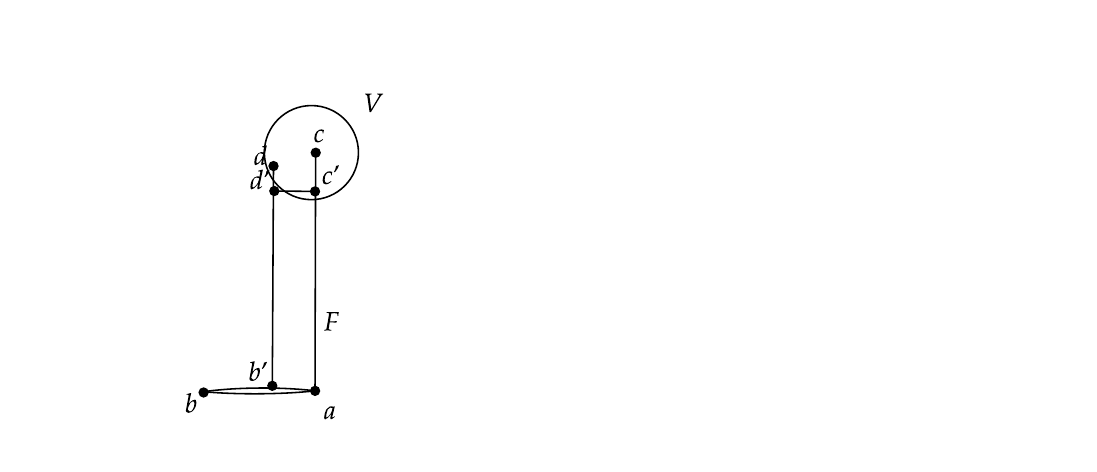}
\end{center}
\caption{Producing a local square at $a$.}
\label{NearbySquarePic}
\end{figure}
   
   We may assume that $V$ is a convex neighborhood of $c$ disjoint
   from $[a,b]$. Choose $d\in V\setminus \rho^{-1}(A)$. Set $b'=\rho(d)$. Note
   that $b'\not=a$ by our choice of $z\not\in F$. We now have a 4-tuple
   $\{a,c,d,b'\}$ satisfying the conditions of the Double Projection
   Lemma \ref{DoubleProjection}: we have that $db'a=b'$ and $cb'a=a$. 
   Thus, setting $c'=dac$ and $d'=cb'd$, Lemma \ref{DoubleProjection} tells us that the 4-tuple
   $\{a,c',d',b'\}$ is  
   a median square. Note that since $V$ is convex, $[d,c]\subset V$, so that $c'\in V$. Also,  $\{a,c',d',b'\}$ are distinct
   by our choice of $V$.
   Thus we have a median
   square $\square(a,c',d',b')$ with $c'\in A\cap V$ and $b'\in [a,b]$ as
   required. 
   \end{proof}

    \paragraph{Example}
    To see that the setting of ER homology manifolds is relevant, even
    when we are dealing with a median structure on $\Rbb^n$, we
    describe an example of a CAT(0) cubulation of $\R^5$ whose
    hyperplanes are ER-homology manifolds, but not manifolds, and
    likewise for orthogonal manifolds.

Let $\Sigma$ be a nonsimply-connected homology 3-sphere with a flag
triangulation. Let $M=\Sigma\times [0,\infty)/\Sigma\times 0$ be the
  open cone on $\Sigma$. Then $M$ is not a manifold since the link at
  0 is $\Sigma$, but it is a homology 4-manifold, and it's a manifold
  away from 0. It also has a CAT(0) cubing.

  Now let $X=M\times \R$. This is a 5-manifold away from $0\times\R$,
  and near $0\times\R$ it is locally homeomorphic to the double
  suspension $\Sigma*S^1$ of $\Sigma$ near the bad circle. So by the
  Cannon-Edwards Double Suspension theorem it is a
  5-manifold. It is also simply-connected at infinity so it is
  homeomorphic to $\R^5$ (Stallings for PL manifolds, and Kirby-Siebenmann
  show that contractible manifolds of dimension $\geq 5$ have a PL
  structure).

  \section{Median algebra structures on $\Rbb^2$}
  \label{DimensionTwo}

    In this section we assume $M$ is a topological median algebra homeomorphic to $\Rbb^2$. We aim to prove the local cubulation results in this separately, firstly because dimension 2 is easier to visualize, but also because it can be done using direct arguments, not relying on the homology manifold results of the previous section. Logically speaking, one can skip this section and go directly to the next section. 
    
    As usual, we will use the term \emph{median square}, denoted  $\square( a,b,c,d)$, to 
  refer to a quadruple of points with its cyclic median structure and the term \emph{square}, denoted $\blacksquare (a,b,c,d)$, to refer to the convex hull of 
  $\{a,b,c,d \}$. Note that by Theorem \ref{EmbeddedCube}, a median square can be ``filled in" to a square in a canonical way,  as a median embedding of $[0,1]\times [0,1]$ (with its canonical product median structure) into $M$.  
  
  By Lemma \ref{CubeIntersection}, we now that two squares intersect in a subsquare of both. Note that this intersection can be a degenerate square, so that squares can meet along a subsegment of an edge or at a vertex, or not at all.

  By a \emph{local squaring} at a point $a\in M$, we mean a
  neighborhood of $a$ which is the union of a finite collection of
  squares all containing $a$ as a vertex. Note that by the above trichotomy regarding the intersection of squares, we may assume that the squares in the local squaring meet along additive intervals (contained in edges of the squares) or $\{a\}$.

  \begin{thm}[Local squaring]
 Let $M\cong\R^2$ be equipped with a  median structure. Then every point of $M$ has a local CAT(0) squaring.
   \label{LocalSquaring}
   \end{thm}

   The proof of this theorem involves an understanding of the additive intervals that emanate from a given point $a$. There are four main steps. 
   \begin{enumerate}
   \item Show that there are non-trivial additive intervals emanating from $a$. 
   \item Show that each such additive interval extends to an infinite ray emanating from $a$.
   \item Bound locally the number of additive intervals emanating from $a$.
   \item Show that any two cyclically consecutive additive intervals emanating from $a$ are edges of a square.
   \end{enumerate}

   \paragraph{Step 1.} 
   The first thing to recall that the dichotomy regarding the convexity of
   the half-open interval $(a,b]=[a,b]\setminus \{a\}$.
   
   When $(a,b]$ is not convex, by Lemma \ref{SquareInInterval}, we obtain a non-degenerate square in $[a,b]$ with a vertex at $a$. Since by Lemma \ref{Bromberg}, a square has a product structure and we are in dimension 2, the sides of the square are additive intervals.  This provides us with two additive intervals emanating from $a$.

   This, we are left to analyze the case that $(a,b]$ is convex, which is more involved. We will have to analyze more closely the structure of $\rho^{-1}(a)$ and its frontier, where $\rho$ is the projection map to $[a,b]$.  
   For the next few lemmas we fix some notation:
   \begin{itemize}
   \item $\rho:X\to [a,b]$, the usual gate projection map
   \item $A=\rho^{-1}(a)$
   \item $U=\rho^{-1}((a,b])$
   \item $\overline U$, the closure of $U$.
   \item $F=A\cap \overline U$, the frontier of $A$. 
   \end{itemize} 
   
   As noted in Lemma \ref{Fset}, the sets $U, \overline U, A$ and $F$ are all convex. We now claim that they are unbounded.   
   \begin{lemma}
    Let $[a,b]\in M$ be an interval suppose that $(a,b]$ is convex. Let $\rho:M\to [a,b]$ the gate map to $[a,b]$ then $A, U$ and $F$ are unbounded. 
   \label{UnboundedPreimage}

   \end{lemma}
      \begin{proof}  If $A$ were bounded, we would have that $M-A$ is not contractible (because the first homology would not be trivial). But if $(a,b]$ is convex, then $U=\rho^{-1}((a,b])$ is convex and hence contractible, a contradiction. Similarly, $U$ is unbounded, for otherwise $A$ would not be contractible. Finally, if $F$ were bounded (and hence compact, since it is closed), we would have that $\R^2$ has more than one end. 
   \end{proof}

From the above, we are ready to prove Step 1.

   \begin{prop}
   Suppose that $a\in M$. Then there exists an additive interval emanating from $a$. 
   \end{prop}
   
   \begin{proof}
   Consider some interval $[a,b]$. If $(a,b]$ is not convex, then we apply Lemma \ref{SquareInInterval}, to conclude that there exists a square $\square(a,c,d,e)\subset [a,b]$. Now note that $[a,c]$ is an additive interval and we are done. 
   
   So suppose that $(a,b]$ is convex.  Then we apply Lemma \ref{NearbySquare} to conclude that there exists a square $\square{(a,w,y,z)}$ with $a$ as a vertex. As before, note that $[a,z]$ (and $[a,w]$) is an additive interval.   
  
        \end{proof}

\paragraph{Step 2.} We wish to show that every additive interval emanating from $a$ extends to an infinite additive ray. 
 
To obtain this, we show that when $(a,b]$ is convex, the set $F$ is actually an embedded real line. 

Establishing the analogue of this in higher dimensions relies on the theory of  homology manifolds and ENR's (see Section \ref{Fancy}). However, in the 2-dimensional case, there is a more direct argument, which we give here.

\begin{prop}
  Let $[a,b]$ be an interval with $(a,b]$ convex. Then $F$ is closed, convex and homeomorphic to $\R$. 
   \label{TwoDimFset}
   \end{prop}

\begin{proof}
  This follows immediately from Theorem \ref{mitchell} but we give a
  more elementary proof.
Since $F$ inherits a median structure, it is an ER and it suffices to show that $F$ is a
1-manifold. Let $x\in F$ be an arbitrary point. By Corollary
\ref{RankAndConvexityEuclidean} there is a compact neighborhood $N$ of
$x$ in $F$. Since $N$ is a 1-dimensional ER, it is connected, locally
connected and contains no circles, and therefore it is a dendrite
\cite{whyburn}. We now argue that $x$ has valence 2. If $x$ has
valence $>2$ there is a disk neighborhood $D$ of $x$ in $\R^2$ whose
intersection with $F$ includes a tripod $T$: 3 arcs that have $x$ as a common
endpoint, the other endpoints are on $\partial D$, and any two arcs
intersect only at $x$. There are points of $U$ arbitrarily close to
$x$ in at least two components of $D\smallsetminus T$, for otherwise
points on one of the arcs would not be in $\overline U$. But then the
interval between two points of $U$, being connected, would have to
intersect $T$, contradicting the convexity of $U$.

Now we know that $F$ is a 1-manifold, possibly with boundary, and we
need to rule out the possibility that $F$ is a point or homeomorphic to $[0,1]$
or $[0,\infty)$. The first two cases are excluded since then
  $U=\R^2\smallsetminus F$ would not be contractible. The last
  possibility gives a contradiction similar to above, by taking two
  points of $U$ in distinct components of $D\smallsetminus F$ where
  $D$ is a disk neighborhood of a point $x$ that is not the endpoint.
\end{proof}

Recall that, as mentioned in Section \ref{Fancy}, the set $F$ defined above is called the orthogonal manifold to $[a,b]$.
   
   \begin{prop}
 Every additive interval extends to an infinite additive ray. 
   \label{ExtendibleR2}
   \end{prop}
   
   \begin{proof}

 Let $[a,b]$ be an additive interval. We will show that $[a,b]$ extends to an infinite closed convex ray based at $b$.  
 
 As before, let $\rho:M\to [a,b]$ be the gate map. Since $[a,b]$ is additive, we have that $(a,b]$ is convex (otherwise, by Lemma \ref{SquareInInterval}, there would be a square embedded in $[a,b]$). Let $F_{ab}$ be the orthogonal manifold to $[a,b]$ at $a$. Then by Lemma \ref{NearbySquare}, there exists a square $\square(a,x,y,z)$, with $x$ on the orthogonal manifold and $z\in [a,b]$. Now we consider the additive interval $[a,x]$ and the orthogonal manifold $F_{ax}$ to $[a,x]$ at $a$. Note that $F_{ax}$ is a closed real line that agrees with $[a,b]$ along the subinterval $[a,z]$. Thus the component of $F_{ax}-z$ containing $a$ together with $[a,b]$ is the required closed convex ray based at $b$. 
    \end{proof}

 \paragraph{Step 3.} We wish to show that the number of additive intervals emanating from $a$ is locally bounded. This holds in greater generality, so we state it in that form here. Let $K\subset M$ and $a\in K$. We say that an additive interval $[a,x]$ based at $a$ is \emph{ full in $K$}, if it is the intersection with $K$ of an additive ray based at $a$. Note that if every additive interval is extendible to an additive ray and $K$ is convex compact, then every additive interval in $K$ at $a$ extends to a full additive interval at $a$. 
  
 More precisely, we aim to show.
 
 \begin{prop}
 Let $M$ be a locally compact, finite rank, connected topological median algebra,
 such that every additive interval extends to an additive ray. Given
 $a\in M$, there exists a compact convex neighborhood $K$ of $a$ such
 that the number of full additive intervals based at $a$ in $K$
 is finite. 
 \label{LocallyBoundedAdditive}
 \end{prop}  
 
\begin{proof}

 Recall that by Theorem \ref{LocallyConvex}, $M$ is locally convex. 
 Note that for each additive ray emanating from $a$, the intersection
 with a compact convex neighborhood of $K$ is a full additive
 interval. Every such interval is of the form $[a,x]$ where $x\in
\partial K = K\setminus \mathring{K}$.

Observe that by local convexity and local compactness, we can arrange that there are a pair of compact convex neighborhoods $K$ and $L$ such that $a\in K\subset\mathring L$. 

Note that each full additive interval in $K$ extends to one in $L$. Moreover, it $[a,x]$ and $[a,y]$ are two different full additive intervals in $K$, they extend to full additive intervals $[a,x']$ and $[a,y']$ in $L$ respectively, such that the intersection of $[a,x']$ and $[a,y']$ is a subinterval $[a,z]$, where $z$ is in $K$. 

Now suppose that we have an infinite sequence $\{[a,x_n]\}$ of full
additive intervals in $K$. We extend each $[a,x_n]$ to a full additive
interval $[a,x_n']$ in $L$. By compactness of $L$, we may pass to a
subsequence such that $\{x_n'\}$ converges.

Note that $x_n'\to x'$, where $x'\in \partial L$, which contains a neighborhood not meeting $K$. But by choosing $x_n'$ and $x_m'$ sufficiently close to $x'$, we have that $ax'_nx'_m\to x$ which means the intervals $[a,x'_n]$ and $[a,x'_m]$ bifurcate at points close to $x'$, contradicting the fact that they bifurcate in $K$. 
\end{proof}

     \paragraph{Step 4.} 
     We now have at $a\in M$ a compact convex neighborhood $K$ and a
     locally finite collection of additive intervals $[x_1,a],\ldots,
     [x_n,a]$ in $K$, so that $[a,x_i]\cap [a,x_j]=\{a\}$ and 
     so that every additive interval emanating from $a$ is either contained in some $[a,x_i]$ or contains some $[a,x_i]$. 
     We label them cyclically about $a$, so that $[a,x_i]$ and $[a,x_{i+1}]$ do not have any additive intervals emanating from $a$ between them.

     We now prove 
     
     \begin{claim}
     There exist subintervals $[a,y_i]$ of $[a,x_i]$ and points $z_i$
     such that for all $i$ (mod $n$), the 4-tuple $\{a,y_i, z_i,
     y_{i+1}\}$ is a median square. 
     \end{claim}
     
     This claim will complete the proof of Theorem
     \ref{LocalSquaring}, as we may then pass to subsquares of these
     squares so that any two squares which intersect along a piece of
     their edge actually intersect along their entire edge. The union of these squares would then be a neighborhood of $a$ by Invariance of Domain. 
          
     To prove the claim, consider an arc $\alpha$ in $K\setminus
     \{a\}$ between $x_i$ and $x_{i+1}$ which does not pass through
     any of the other additive intervals emanating from $a$. Consider
     the gate map $\rho:M\to [a,x_i]$ and let $A=\rho^{-1}(a)$. Note
     that $\rho(\alpha)$ is a path in $[a,x_i]$ from $x_i$ to
     $a$. Consider the first point $w_i\in \alpha$ which projects to
     $a$. By Lemma \ref{NearbySquare}, since $w_i\in A\setminus
     \mathring{A}$, we have  a square of the form $\square(a,
     w'_i, u_i, v_i)$, where $w'_i\in A$ and close to $w_i$ and
     $v_i\subset [a,x_i]$. This establishes the claim.

  \section{The Local Cubulation Theorem in Higher Dimensions}
  \label{HigherDimensions}

  In this section, we generalize the local squaring result of the previous section to higher dimensions. The proof uses some of the same ideas and follows the same outline, but also relies on some new ingredients. 
  
  As in the 2-dimensional case, we will want to look at an interval
  $[a,b]$ and under the assumption that $(a,b]$ is convex, 
    use the orthogonal manifold to $[a,b]$ at $a$ to build
    cubes around $a$. Recall that from Proposition \ref{Fset}, the orthogonal manifold to $[a,b]$ at $a$ is
    an ER homology manifold of one lower dimension, so we will make an
    inductive argument.

  Our main theorem is 
  
\begin{thm}
Suppose that $M$ is a topological median algebra and an $n$-dimensional ER homology manifold. Then  every point $a\in M$  is locally isomorphic to a CAT(0) cube complex as a topological median algebra. 
\label{LocalCubingHigherDimension}
\end{thm}

\begin{proof}

We will show this in steps, as in the 2-dimensional case. 
\paragraph{Step 1.}  We show that there exist additive intervals
emanating from $a$. As in the 2-dimensional case, we focus on
nontrivial intervals of the form $[a,b]$. If $(a,b]$ is not convex, by Lemma \ref{SquareInInterval}, we
  find a median square $\square(a,c,d,e)\subset [a,b]$. By Lemma \ref{Bromberg}, for  $\square(a,c,d,e)$, we have that $[a,d]\cong [a,c]\times [a,e]$, so that $rk[a,d]=rk[a,c]\times rk[a,e]$. Thus, the interval $[a,c]\subset[a,b]$ is an interval emanating from $a$ of smaller rank.  Since $M$ has  finite rank, by applying Lemma \ref{SquareInInterval} finitely many times, we eventually
  obtain a nontrivial interval $[a,b]$ with $(a,b]$ convex.

As usual, we let $\rho_{ab}$ denote projection onto $[a,b]$. We define $U=\rho^{-1}((a,b])$ (which is open in $M$), and define $F=\overline{U}\setminus U$, the orthogonal manifold to $[a,b]$ at $a$. 
By Proposition \ref{Fset}, we have that $F$ is an ER homology manifold of dimension $n-1$. Thus, by induction, it has a local cubulation by $(n-1)$-cubes at $a$. In particular, this means that there are additive intervals emanating from $a$ in $F$ and hence in $M$. 

\paragraph{Step 2.} 
We show additive intervals emanating from $a$ are extendible to infinite additive rays. We proceed as we did in dimension 2. We consider an additive interval $[a,b]$ at $a$. We then consider the orthogonal manifold $F$ to $[a,b]$ at $a$. Since by induction it is locally cubulated, there exists an additive interval $[a,c]$ in $F$. Now since by definition $\rho_{ab}(F)=a$, we have that $abc=a$. This also tells us that the projection of $b$ onto $[a,c]$ is $a$. Thus, $b$ is contained in the orthogonal manifold $F'$ to $[a,c]$ at $a$. Since $F'$ is an ER homology manifold of dimension $n-1$, it follows, by induction that $[a,b]$ is extendible to an infinite additive ray in $F'$ and hence in $M$, as required.

\paragraph{Step 3.} 
Bounding the number of additive intervals emanating from $a$. This is direct application of Proposition \ref{LocallyBoundedAdditive} as $M$ is locally compact and additive intervals are extendible to additive rays. 

\paragraph{Step 4.} 
Finally we produce the requisite local
cubulation. 

Recall that by a $k$-cube in $M$ we mean a
 median embedding of $[0,1]^k$ in $M$ spanned by $k$ additive  intervals. 

 By Lemma \ref{CubeIntersection},  any two cubes at $a$ meet along a subcube of a face of each (where the term ``face" includes the entire cube, not just proper faces). We will need to take a bit of care to find cubes that actually meet along faces and not subcubes of faces.

  Since there are finitely many (infinitely extendible) additive rays
  emanating from $a$, and cubes are spanned by additive intervals,
  there are finitely many germs of cubes at $a$; that is, there exists a finite collection of cubes $\C_a$ at $a$, such that for every $k$-cube $\sigma$ at $a$, there exists a $k$-cube $\tau\in \C_a$ such that $\sigma$ and 
  $\tau$ overlap along a proper $k$-subcube of both.

  Since $\C_a$ is finite, for each additive ray $J$ emanating from $a$, we can choose a point  $b_J\in J$ so that the additive interval $\nu_J=[a,b_J]$ is contained in every $\sigma\in\C_a$ which contains some proper subinterval of $J$. We can then replace each $\sigma\in\C_a$ by a cube spanned by a collection of additive intervals in $\I= \{\nu_J\}$. Note now that when such cubes intersect, they intersect along a face of both, so that $K_a = \bigcup_{\sigma\in\C_a}  \sigma$ is a cube complex embedded in $M$. We also have the following

  \begin{obs}
  If some collection of additive intervals at $a$ spans a cube, then the corresponding additive intervals in $\I$ span a cube. 
  \end{obs}
 
 In order to apply Proposition \ref{LocallyEverything} to conclude that $K_a$ is a neighborhood of $a$,  and to see that $K_a$ is CAT(0) we will need to prove the following about $K_a$. 
 
 \begin{enumerate}
 \item Every cube in $K_a$ containing $a$ is contained in an $n$-cube.
 \item  Any $(n-1)$-cube in $K_a$  containing $a$ is contained in precisely two
   cubes.
 \item  The link of $K_a$ at $a$ is flag.
\end{enumerate}
 We will prove these three properties by induction on dimension.

 For Item 1, consider a $k$-cube $\sigma$ in $\C_a$.
 Consider one of
 its edges $[a,b]\in \C_a$. Let $F=F_{ab}$ 
 denote the orthogonal manifold
 to $[a,b]$ at $a$. By induction, there is a local cubulation
 satisfying the above three properties for $F$ at $a$. Let $\C_a(F)$
 denote the set of cubes of this local cubulation.
 Note that $\sigma$
 decomposes as a product $\sigma\cong\tau\times [a,b]$, where
 $\tau\subset F$. Now since $F$ satisfies Property 1 above, there
 exists an $(n-1)$-cube $\nu\in \C_a(F)$ such that
 $\tau\subset\nu$. Consider all the edges of $\nu$ at $a$. By Lemma \ref{NearbySquare}, each edge of $\nu$ has a subinterval which spans a square with some subinterval of $[a,b]$. By the Observation above, we have that the collection of edges of $\nu$ together with $[a,b]$ have the property that every pair spans a square. Thus, by Lemma \ref{Flag}, all the edges of $\nu$  together with $[a,b]$ span an $n$-cube $\nu\times [a,b]$  containing $\tau$.

For Item 2, the argument is similar. We let $\sigma$ be an $(n-1)$-cube  and let $[a,b]$ and $\tau$ be as before, so that $\sigma\cong\tau\times [a,b]$ and $\tau\in F$. Now $\tau$ is an $(n-2)$-cube in $F$, and $F$ is $n-1$-dimensional. So by induction, there exist precisely two $(n-1)$-cubes $\nu_1$ and $\nu_2$ in $\C_a(F)$ containing $\tau$. Applying Lemma \ref{Flag}, to the edges of $\nu_1$ and $\nu_2$ together with $[a,b]$, we see that there are two $n$-cubes $\nu_1\times [a,b]$ and $\nu_2\times [a,b]$ containing $\sigma$. Moreover, any $n$-cube containing $\sigma$, would have to have the structure of a product of an $n-1$-cube in $F$ and $[a,b]$, so that there are exactly two $n$-cubes at in $K_a$ containing $\sigma$. 

Finally, we have to show Item 3, which is that the flag condition is satisfied. For this, we need to show that if $[a,b_1], \ldots,[a,b_k]\in \C_a$ are a collection of additive intervals such that every two bound a square, then they span a $k$-cube. This follows immediately from Lemma \ref{Flag}.

We thus obtain the $K_a$ is a CAT(0) cube complex.  Now the star of $a$ in $K_a$ is  a simplicial complex $P$ to which we may apply Proposition \ref{LocallyEverything} to conclude that $K_a$ contains some open neighborhood $U$, as required. 

\end{proof}

\section{Metrization}
\label{Metrization}

 \subsection{Complete metrization}
 Recall that a median metric space is a metric space $(X,d)$ such that
 for any three points $a_1,a_2,a_3\in X$, there exists a unique point
 $x$ such that $d(a_i,a_j)=d(a_i,x)+d(x,a_j)$ for any
 $i\not=j\in\{1,2,3\}$. It is an exercise to check that if we define
 the interval between two points as $$[a,b]=\{x\in X\vert
 d(a,b)=d(a,x)+d(x,b)\}$$ then for a median metric space, these
 intervals satisfy the properties of intervals in Sholander's
 Theorem. Thus a median metric space is  also a median algebra. It is
 also easy to see that the median is a continuous map, so that a
 median metric space is naturally a topological median algebra. (For details, see, for example \cite{ChatterjiDrutuHaglund2010}.)

We now wish to understand when the median structures on homology manifolds are completely metrizable. That is, given a median structure on an ER homology manifold, when can we put a complete median metric on it realizing the given median structure.

Recall that a topological median algebra is \emph{interval compact} if the interval between any two points is compact.  Our main theorem is the following

\begin{theorem}
If $M$ is a  median structure on an $ER$ homology manifold, then $M$ is completely metrizable if and only if it is interval compact. 
\label{CompleteMetrization}
\end{theorem}

It is not difficult to see that when $M$ is a complete, finite rank median metric space, intervals are compact (see \cite{Bowditch24}). 

It is possible for a median structure on $\R^2$ to have non-compact intervals and hence not completely metrizable. 

\begin{example}
Consider the standard $\ell_1$-median structure on the plane $\R^2$. Let $X=\R^2\setminus \{(x,y)\vert x\leq 0, y\leq 0\}$, the plane minus one closed quadrant. Then it is easy to check that $X$ is a median subalgebra, so that it inherits a median structure from $\R^2$. Moreover, the interval $[(1,-1),(-1,1)]$ is not compact (since its intersection with $X$ is not closed as a subset of $\R^2$). 
Thus, since a complete median metric space has compact intervals,  this median structure is not completely metrizable.  One can see directly that $X$ is not completely metrizable. Consider the infinite ray based at  $(-1,1)$ and going towards
$(-1,0)$.  A complete metric on $X$, would have have to assign this ray infinite length. However, it is parallel to the half open interval  $[(1, 1), (1, 0) )$, which has finite length. 
\label{NonCompactExample}
\end{example}

Thus,  the work in proving Theorem \ref{CompleteMetrization}  will be to show that an $ER$ homology manifold which is interval compact is completely metrizable. 
For the rest of this section, we will thus assume that $M$ is an interval compact median structure on an $ER$ homology manifold. 

We will need some basics about convex hulls in this setting. Let
$S\subset M$. The convex hull of $S$ is the intersection of all convex
sets containing $S$. Recall (Corollary \ref{CubeIsAnInterval}) that the convex hull of a median cube of
dimension $n=\dim M$ is an embedding of a real cube in $M$, since it is an
interval between diagonal points.  \mb{is $n=\rank$?, how do we know
  it's a cube? }\msb{I actually don't know why we need this. I think I will just scratch it}

The convex hull can be defined ``from below" as follows (see \cite{Bowditch24}). Given a subset $S\in M$, we define the \emph{join of $S$}

$$J(S)=\bigcup_{a,b\in S} [a,b]$$

We  define the iterated join recursively as $J^n(S)=J(J^{n-1}(S))$.

By \cite{Bowditch24}, we have that if  $S$ is finite and $M$ is a finite rank median algebra, then there exists $n$ such that $Hull(S)=J^n(S)$.
 
\begin{lemma}
Suppose that $K\subset M$ is compact and $M$ is path connected,
finite rank, locally compact and interval compact. Then $J(K)$ is
compact.

\label{JoinCompact}
\end{lemma}

\begin{cor}
When $M$ is a finite rank, path connected, locally compact and interval compact
median algebra, the convex hull of a compact set is compact.
\label{HullCompact}
\end{cor}

\begin{proof}
Consider an open cover $\cU$ of $J(K)$. We wish to show $\cU$ has a finite subcover.  By assumption, $M$ is locally compact, so we may assume all the elements of $\cU$ have compact closure. 

For any interval $[a,b]$ with $a,b\in K$, we have a cover of it by elements of $\cU$ and since $K$ is compact, there exist finitely many such $U_1,..., U_n$. We let $U_{ab}=\bigcup_{i=1}^n U_i$, so that $[a,b]\subset U_{ab}$ and $\overline U_{ab}$ is compact.

We now need the following 

\paragraph{Claim.} For each $[a,b]$, there exist neighborhoods $U_a$ of $a$ and $U_b$ of $b$ such that for any $x\in U_a$ and $y\in U_b$, $[x,y]\subset U_{ab}$. 

To prove the claim, suppose that there do not exist such
neighborhoods. Then we may find a sequences in $(x_n)\in M$ and
$(y_n)\in M$ 
such that $x_n\to a$ and $y_n\to b$ and $[x_n,y_n]\not\subset U_{ab}$. Since each $[x_n,y_n]$ is path connected, we obtain a point $z_n\in [x_n,y_n]\cap B(U_{ab})$, where $B(U_{ab})=\overline U_{ab} - U_{ab}$ is the boundary of $U_{ab}$. Since $B(U_{ab})$ is compact, we may pass to a subsequence and so that there exists  
$c\in B(U_{ab})$ such that $z_n\to c$. Now we obtain $z_n=x_ny_nz_n\to abc$, but $abc\in [a,b]$ whereas $z_n$ lies outside some open neighborhood of $[a,b]$, a contradiction. This proves the claim. 

From the claim we now obtain for every interval $[a,b]$ in $J(K)$ open sets $U_a$ and $U_b$ such that all intervals with endpoints in $U_a$ and $U_b$ are contained in $U_{ab}$. The collection of open sets $U_a\times U_b$ running over all pairs of points in $K$ forms an open cover of $K^2$. Since $K^2$ is compact, there exists a finite subcover $U_{a_1}\times U_{b_1}...U_{a_n}\times U_{b_n}$ of $K^2$. The open sets $U_{a_ib_i}$ for $i=1,..,n$ then form an open cover of $J(K)$ and each $U_ {a_ib_i}$ is covered by finitely many elements of $\cU$, so we have a finite subcover of $\cU$, as required. 
\end{proof}

Now we use this to show

\begin{lemma}
If $M$  is an ER homology manifold with an interval compact median structure, $M$ has an exhaustion by compact convex sets. 
\label{ConvexExhaustion}
\end{lemma} 
  
  \begin{proof}
  We already know that $M$ has an exhaustion by compact sets. Thus, we need to show that for each compact set $K$, there exists $L$ compact and convex with $K\subset L$. But by Lemma \ref{HullCompact}, $\hull(K)$ is compact. 
\end{proof}
  
  Terminology: when we metrize a  median structure, we turn each $n$-cube in the median structure into a product of $n$ compact intervals in $\Rbb$ (of possibly different lengths), with the $\ell_1$ metric on this product. We call such an object an \emph{n-box} or simply a \emph{box}. 
     
  Roughly speaking, in order to completely metrize $M$, we proceed by describing $M$ recursively as an ``onion" built out of compact layers, where the first layer is the star of a vertex, and the $n$'th layer is the collection cubes that meet previous layer.  We will metrize in such a way that each layer has thickness one. This will ensure that the final median metric space we obtain is complete. 
  
More precisely,  the main theorem of this section is as follows. 
  
  \begin{thm}
  There exists an exhaustion $K_1\subset K_2\subset\ldots$ of $M$ by compact convex sets and a median metric on each $K_i$, such that 
  
  \begin{enumerate}

  \item for each $i$, $K_i$ admits a median metric coming from a
    finite CAT(0) cubulation such that the restriction of the metric to each
    cube is an $\ell_1$ metric on a  box,
   \item for each $i<j$, the metric on $K_j$ is an extension of the metric on $K_i$, obtained by isometrically gluing $n$-boxes along $B(K_i)$, where $B(K)=K\setminus \mathring{K}$ denotes the frontier of $K$ 
   \item for each $i$, $d(B(K_{i+1}), K_i)\geq 1$. 
  \end{enumerate}
  \label{ThickExhaustion}
  \end{thm}
  The basic step in the above theorem is what we call the  ``1-thickening" and it goes like this.

  \begin{figure}[h]
\includegraphics[scale=.7]{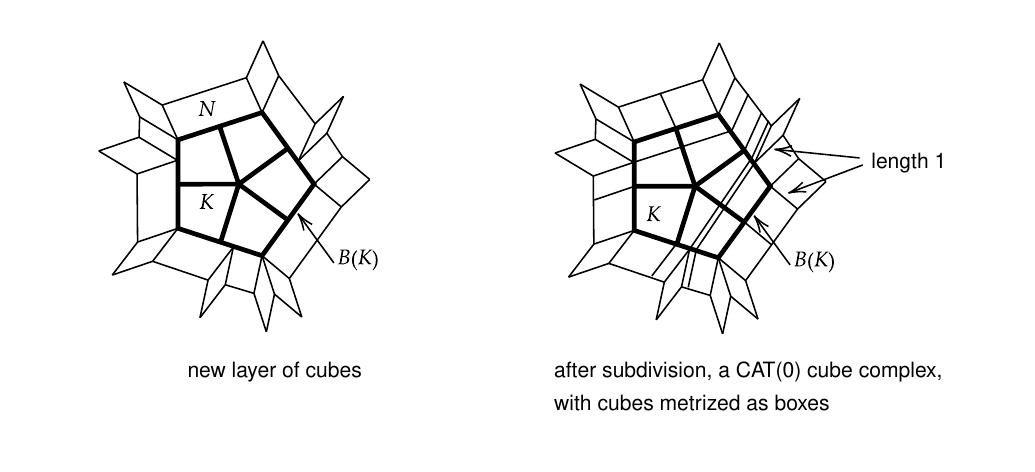}
\caption{The basic thickening step. For each new cube, edges going out from $K$ are assigned length 1.}
\label{Thickening}
\end{figure}

 \paragraph{The 1-thickening construction.} Refer to Figure
 \ref{Thickening} for this discussion. Consider a compact convex set
 $K$.  First, observe that a convex subset of a cube is a cube
 (perhaps of lower dimension). Since $K$ is covered by finitely many
 cubes $C_1,...C_n$, it follows that $K$, after replacing each $C_i$ with its intersection with $K$, has the structure of a finite CAT(0) cube complex, except that cubes in adjoining neighborhoods may not ``match up" along their faces (see, for example, the first diagram in Figure \ref{Thickening}). To make this into an actual CAT(0) cube complex, we need to subdivide. To this end, let $V$ denote the collection of all vertices of the $C_i$'s.
 A theorem of Bowditch \cite[Proposition 3.3.3]{Bowditch24}, tells us that any finite subset of a finite rank median algebra generates a finite subalgebra. Let $\hat V$ denote the median subalgebra generated by $V$. Note that by a result of Roller \cite{Roller98}, 
 $\hat V$ is the vertex set of a CAT(0) cube complex $Y$. There exists a natural embedding of $Y$ in $M$ extending the inclusion of $\hat V\hookrightarrow M$. Note that all the vertices of $V$ are vertices of $Y$ and that the cubes of $Y$ are subcubes of the $C_1,...,C_n$. Since $K$ is convex, and $V\subset K$ and $K$ is the union of the the cubes spanned by $V$, we see that $Y$ is an honest CAT(0) cubulation of $K$, as required. \msb{Added an explanation of why we get an actual CAT(0) cubulation}
 
Now we assume that the cubes of $K$ have been metrized as boxes. Consider $B(K)$, the frontier of $K$. We let $N=N(K)$ denote the union of the cubes that meet $K$. We wish to metrize $N(K)$ so that it agrees with the metric on $K$, and so that the distance between the frontier $B(N)$ of $(N)$ and $K$ is at least 1.

Note that in a CAT(0) cube complex $X$, any 1-skeleton geodesic between vertices crosses the same hyperplanes, a median metric on $X$ consistent with the CAT(0) cubical structure is given by an assignment of a positive number to each hyperplane. We are assuming we have already metrized $K$, so positive numbers have been applied to the hyperplanes of $K$. However, in producing the  CAT(0) cubical structure on $N$ we may have subdivided $K$. Thus, if a given hyperplane $\hh$ in $K$ has the value $\lambda$, then for the subdivided structure, and there are $n$ parallel copies of $\hh$ in the subdivision, then each such copy gets assigned a value of $\lambda/ n$, so that the metric stays the same. 

Now we address the rest of the hyperplanes in $N$. These are hyperplanes that do not meet $K$. We assign each such hyperplane the value $1$. Since every 1-skeleton path between $B(N)$ and $K$ must cross one of these hyperplanes, we obtain the desired condition that $d(B(N), K)\geq  1$.

 With this basic construction in mind, we are now ready to prove the theorem.
 
\begin{proof}[Proof of Theorem \ref{ThickExhaustion}] We consider an exhaustion $L_1\subset L_2\subset\ldots$ of $M$ by compact, convex sets.

We let $K_1=L_1$. Since $K_1$ has a finite cubulation, we metrize it as an $\ell_1$ metric in which each edge of each cube has length 1. 

Now we consider $K_1\cup L_2$. Let $K'_2=\hull(K_1\cup L_2)$. Since $K_1$ and $L_2$ are compact, 
by Lemma \ref{HullCompact}, $K'_2$  is compact. 

 As noted in the 1- thickening construction, $K'_2$ admits a cubulation extending the one on $K_1$, we give each new cube any $\ell_1$ metric we like which extends the metric on $K_1$. 

Now we apply the 1-thickening constructing to $K'_2$. Let $K_2=N_1(K'_2)$. From the properties of the 1-thickening, we have that $d(B(K_2), K'_2))=1$. Since $K_1\subset  K'_2$, we have that $d(B(K_2), K_1)=1$.

We continue in this manner to produce $K_n$ such that $K_{n-1}\subset K_n$, $L_n\subset K_n$ and $d(B(K_n), K_{n-1})=1$. 
  \end{proof}
  
Finally we obtain

\begin{claim}
The metric constructed above yields a  complete metric on $M$.
\end{claim}

\begin{proof}
Since the metric above is constructed via an exhaustion by convex sets $K_1\subset K_2...$, it follows that any two points of $M$ are contained in some $K_n$, from which we can compute their distance, and this distance is the same for any larger $K_m$. We thus have a metric on $M$. Now suppose that $\{x_n\}$ is a Cauchy sequence. Then there exists $N$, such that $d(x_i,x_j)<1$ for all $i,j\geq N$. Choose $m$ such that $x_i\in K_m$ for all $x_i$, $i=1,\ldots,N$. Then all $x_i\in K_{m+1}$ for all $i>N$, because $d(x_i,x_N)<1$ and points outside of $K_{m+1}$ are distance at least one from $K_N$. 
\end{proof}

\subsection{Constructing examples}

The above metrization construction also gives a way of producing many median structure on ER homology manifolds that are not median isomorphic to CAT(0) cube complexes. One builds a metric median space recursively in layers as in the 1-thickening construction. At each stage, we obtain a median metric space built out of finitely many boxes being glued around the previously constructed compact median space. Since we are allowed to subdivide at each stage, we may add more singular points as we go out. The union of all the layers will be a complete median metric structure. If one adds more and more singular points as one moves out, one can obtain the property that a given square (say in the first layer) will have infinitely many singular leaves (see Figure \ref{NonCubulatedExample}). This is something that does not occur in a CAT(0) cube complex.

  \begin{figure}[h]
\includegraphics[scale=.7]{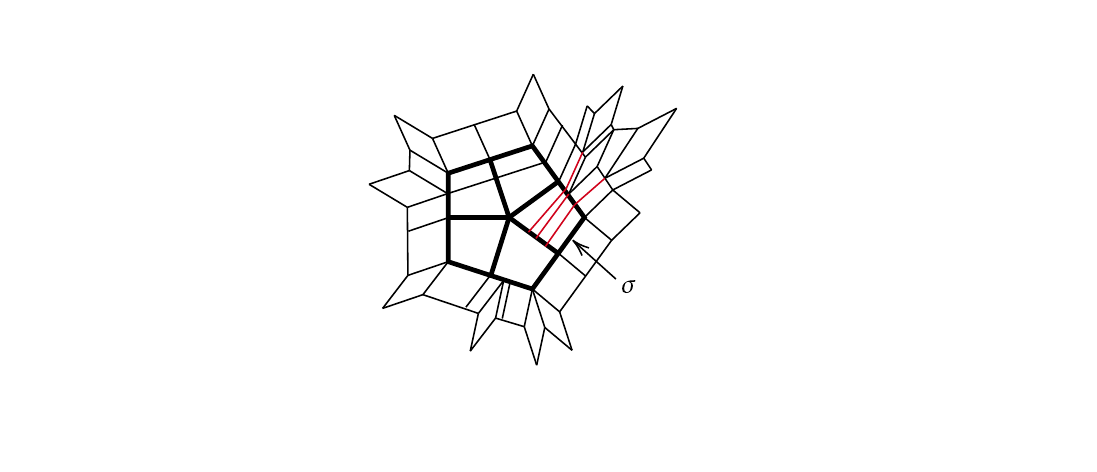}
\caption{Carrying out the 1-thickening construction, subdividing more and more as we go outward. The cube $\sigma$ will contain infinitely many singular leaves.}
\label{NonCubulatedExample}
\end{figure}

\subsection{Completions: a median Floyd construction}

In this section we describe a variant of the construction of the
previous section to construct median structures on the compact ball
$B^n$ for $n>1$, which are not locally cubulable. This stands in
contrast to the main theorem of this paper. One such example, is the metric ball in $\ell_1$ on $\R^2$. Note that the boundary of this ball contains a line interval in the line $y=x$, and points along this interval, do not have cubulated neighborhoods. In this section, we will see that there are in fact uncountably many median structures on the disk for which points on the boundary do not have cubulated neighborhoods. 

Let $M$  be a median structure given by a Gromov hyperbolic CAT(0) cubulation of the $\R^n$.  We proceed to metrize $M$ as in the previous section, writing $M$ as an exhaustion by compact CAT(0) cube complexes $K_i$ (here we do not need to address subdivision at each stage). Now we metrize each layer $K_i$, so that the outward factor of each cube has all of its edges of length $2^{-i}$. This provides a median metrization of our given structure similar to the one in the previous section, but for which the whole metric space is bounded (since the distance between any two vertices is bounded by $\sum_{i=1}^\infty 2^{-i}=1$). We let $\overline M$ denote the completion of $M$ with this metric and let $\partial M=\overline M-M$. 

Note that we have natural retractions $K_{i+1}\to K_i$: recall that
each cube $\sigma $in $K_{i+1}$ but not in $K_i$ splits into
$\sigma=\tau\times\nu$, where the $\tau\subset K_i$ and $\nu$ meets
$K_i$ in a point. We apply the natual retraction $\sigma\to\tau$. Note
that the retraction $K_{i+1}\to K_i$ is distance
non-increasing. Moreover, it moves points less and less as
$i\to\infty$. It is not hard to see that the inverse limit of these
retractions coincides with the metric completion $\overline M$. Each
retraction $K_{i+1}\to K_i$ can be approximated arbitrarily closely by
homeomorphisms, and it follows from Morton Brown's theorem \cite{mort}
that the
same is true for the induced maps $\overline M\to K_i$, and in
particular $\overline M$ is homeomorphic to a ball.

Note that if we pick a generic point in the boundary $B(M)$, it does
not admit a local cubulation. 
This provides a median structure on the disk that  does not admit a local cubulation as in Theorem \ref{LocalCubingHigherDimension}.

\bibliography{ER}

\end{document}